\documentclass{amsart}

\usepackage{amsmath,amsthm,amsfonts}
\usepackage{mathtools}
\usepackage{mathrsfs}

\usepackage{tikz}
\usepackage{hyperref}

\usepackage{comment}

\usepackage[english]{babel}

\theoremstyle{plain}

\newtheorem{thm}{Theorem}

\newtheorem{prop}[thm]{{\bf Proposition}}
\newtheorem{lem}[thm]{{\bf Lemma}}

\newcounter{hyp_counter}
\theoremstyle{definition}

\newtheorem{defn}[thm]{Definition}%[section]

\theoremstyle{remark}
\newtheorem{rem}{Remark}
\newtheorem{question}{Question}

\newcommand{\Id}{\operatorname{Id}}

\newcommand{\GL}{\operatorname{GL}}

\newcommand{\Gr}{\operatorname{Gr}}

\newcommand{\N}{\mathbb{N}}

\newcommand{\R}{\mathbb{R}}
\newcommand{\RP}{\mathbb{R}\operatorname{P}}
\newcommand{\SL}{\operatorname{SL}}

\newcommand{\supp}{\operatorname{supp}}

\newcommand{\Z}{\mathbb{Z}}

\newcommand{\wt}[1]{\widetilde{#1}}

\newcommand{\abs}[1]{\left| #1\right|}

\newcommand{\mc}[1]{\mathcal{#1}}

\makeatletter
\def\blfootnote{\xdef\@thefnmark{}\@footnotetext}
\makeatother

\title[Cocycles Measurably Conjugate to Unipotent]{Cocycles Measurably Conjugate to Unipotent Over Hyperbolic Systems}
\author{Jonathan DeWitt}
\address{Department of Mathematics, The University of Maryland, College Park, MD 20742, USA}
\email{dewitt@umd.edu}
\date{\today}

\begin{document}
\begin{abstract}
We show that if a H\"older continuous linear cocycle over a hyperbolic system is measurably conjugate to a cocycle taking values in a unipotent group, then the cocycle is H\"older continuously conjugate to a cocycle taking values in a unipotent group. More generally, we introduce some natural classes of matrices contained in $\GL(d,\R)$, which we call \emph{Zimmer blocks}. Examples of Zimmer blocks are unipotent and compact subgroups. We show that the same conclusion holds for Zimmer blocks.
\end{abstract}

\maketitle

\section{Introduction}

In this paper we prove a rigidity result for cocycles showing, for example, that if the measurable algebraic hull of a cocycle over a hyperbolic system is unipotent, then the cocycle itself is continuously conjugated to a unipotent cocycle.
\blfootnote{This material is based upon work supported by the National Science Foundation Graduate Research Fellowship under Grant No.~DGE-1746045.} 

We now develop the language to state our main cocycle rigidity theorem. We say that a set of matrices $H$ contained in $\GL(d,\R)$ is a \emph{Zimmer block} of exponent $\lambda$ if there exist compact linear groups $K_1,\ldots,K_k$ such that the matrices in $H$ are all matrices of the form:
\begin{equation}\label{eqn:example_block}
\begin{bmatrix}
e^{\lambda} M_1 & * & *\\
0 & \ddots & *\\
0 & 0 & e^{\lambda} M_k
\end{bmatrix}, 
\end{equation}
 where for each $1\le i\le k$ we have $M_i\in K_i$. In other words, a Zimmer block of exponent $\lambda$ is comprised of block upper triangular matrices where each block on the diagonal is of the form $e^\lambda K_i$ for some fixed $\lambda$ and fixed collection of compact matrix groups $K_i$.
 Zimmer's amenable reduction gives that any measurable linear cocycle over an ergodic measurable transformation is measurably conjugate to a cocycle taking values in a group that has a finite index subgroup whose elements are block upper triangular with the blocks on the diagonal each being scalar multiples of orthogonal matrices \cite[Thm. 9.2.3]{zimmer1983ergodic}.  However, in Zimmer's amenable reduction the diagonal blocks at a given point need not all be scaled by the same scalar, unlike the matrix in equation \eqref{eqn:example_block}.
 Examples of Zimmer blocks are compact and unipotent subgroups of $\GL(d,\R)$. Our main rigidity result is the following.

\begin{thm}\label{thm:Zimmer_block_rigidity}
Let $U_{\lambda}\subset \GL(d,\R)$ denote a Zimmer block of exponent $\lambda$.
Let $\sigma\colon \Sigma\to \Sigma$ be a transitive subshift of finite type equipped with a fully supported measure $\mu$ that has continuous product structure.
 Let $A\colon \Sigma\to \GL(d,\R)$ be a H\"older continuous cocycle.
 Suppose that there exists a $\mu$-measurable function $u\colon \Sigma\to \GL(d,\R)$ such that for $\mu$-almost every $x$,
\[
u(\sigma x)A(x)u(x)^{-1}\in U_{\lambda},
\]
then there exists a H\"older continuous function $u'\colon \Sigma\to \GL(d,\R)$ such that \[u'(\sigma x)A(x)u'(x)^{-1}\in U_{\lambda}.\]
\end{thm}
We refer to the function $u$ appearing in the above theorem as a ``transfer function." Note that we do not conclude that $u'=u$ almost everywhere. In fact, we show in Subsection \ref{subsec:example} that we cannot conclude this because it need not be true.

The following finer result shows that if we have a measurable transfer function between two H\"older cocycles at least one of which takes values in a Zimmer block, then the transfer function coincides almost everywhere with a H\"older continuous transfer function.

\begin{thm}\label{thm:transfer_rigidity}
Let $U_{\lambda}\subset \GL(d,\R)$ denote a Zimmer block of exponent $\lambda$.
Let $\sigma\colon \Sigma\to \Sigma$ be a transitive subshift of finite type equipped with a fully supported measure $\mu$ that has continuous product structure.
Suppose that $A\colon \Sigma\to \GL(d,\R)$ is a H\"older continuous cocycle and that $B\colon \Sigma\to U_{\lambda}$ is also H\"older continuous.
Suppose that there exists a measurable function $u\colon \Sigma\to \GL(d,\R)$ such that for $\mu$-almost every $x$,
\[
u(\sigma x)A(x)u(x)^{-1}=B(x).
\]
Then there exists a H\"older continuous function $u'\colon \Sigma\to \GL(d,\R)$ such that \[u'(\sigma x)A(x)u'(x)^{-1}=B(x),\] and further $u=u'$ $\mu$-a.e.
\end{thm}

The ideas in the proofs of Theorem \ref{thm:Zimmer_block_rigidity} and \ref{thm:transfer_rigidity}
 draw on ideas of Butler appearing in \cite{butler2018measurable} and the approach in \cite{pollicott2001livsic}, respectively. The most significant novelty in this work is our extension of Butler's shadowing argument beyond the conformal setting, which occupies the entirety of Section \ref{sec:periodic_orbits}.

Besides their independent interest as rigidity results, rigidity results for cocycles are useful when studying rigidity of smooth systems.
For example, rigidity of transfer functions plays a role in studying conjugacies between smooth hyperbolic systems. If one has a Lipschitz conjugacy between two smooth hyperbolic systems, then its derivative is a measurable transfer function between the derivative cocycles of those two systems. Hence if one can upgrade the regularity of the transfer function, one can also upgrade the regularity of the conjugacy. For a specific instance of this, see \cite[Sec.~3.2]{gogolev2020local}.
Zimmer's measurable amenable reduction is also used in studying rigidity problems.
For example, in recent work of Damjanovi\'{c} and Xu on rigidity of higher rank Anosov $\Z^k$-actions, those authors use Zimmer's amenable reduction for $\Z^k$-actions and then upgrade the measurable structures they obtain to H\"older continuous ones \cite[Sec.~6.6]{damjanovic2020classification}.

\subsection{Related work on cocycles}

In \cite{butler2018measurable}, Butler draws attention to the following important question:
\begin{question}\label{question:cohomology}
Suppose that $A$ and $B$ are measurably cohomologous. Are they then continuously cohomologous?
\end{question}

Butler shows that this is true when $B=\Id$, improving on earlier work of Pollicott and Walkden \cite{pollicott2001livsic}, which contains a fiber bunching assumption. Sadovskaya \cite{sadovskaya2015cohomology} later showed that if $A$ is fiber bunched and $B$ is uniformly quasiconformal, then a measurable conjugacy between the two is continuous. A further measurable rigidity result was recently obtained by Kalinin, Sadovskaya, and Wang \cite[Thm.~2.2]{kalinin2022local} for conjugacies between a constant cocycle and its perturbation.

Although the question is well understood in the case where one of the cocycles is trivial, understanding the general picture will require a theory that has yet to be developed.
 For example, in \cite[Sec. 9]{pollicott2001livsic} Pollicott and Walkden exhibit a cocycle near the identity of the form 
\[
\begin{bmatrix}
f_1 & f_2\\
0 & 1
\end{bmatrix},
\]
which is measurably but not continuously cohomologous to 
\[
\begin{bmatrix}
f_1 & 0 \\
0 & 1
\end{bmatrix}.
\]
The obstruction to continuous conjugacy is that the former cocycle may have a Jordan block in its periodic data.  
Note that both of the above cocycles are fiber bunched, hence knowing that $A$ and $B$ are close and fiber bunched is insufficient to upgrade the regularity of a measurable conjugacy.
Pollicott and Walkden's example differs substantially from the setting considered in this paper because their cocycle has two distinct Lyapunov exponents.  
Thus it seems that a more subtle theory is needed to fully address this question. 

The proof of Theorems \ref{thm:Zimmer_block_rigidity} draws heavily on ideas of Butler in \cite{butler2018measurable}.
 That paper shows, in the same setting as considered in this paper, that if a cocycle preserves a measurable conformal structure, then it preserves a continuous conformal structure.

Related to Question \ref{question:cohomology} is another.
\begin{question}\label{question:algebraic_hull}
Suppose that $G\subset \GL(n,\R)$ is a closed subgroup. When does it follow that if a $\GL(n,\R)$-cocycle is measurably conjugated into $G$, then it is continuously conjugated into $G$?
\end{question}

The main result we are familiar with in this direction is \cite[Thm. 3.4]{kalinin2013cocycles}, which shows a continuous amenable reduction for fiber bunched cocycles with all Lyapunov exponents zero.
We are curious to know if their fiber bunching assumption can be removed.

\subsection{An example}\label{subsec:example}

A useful example to keep in mind throughout this paper is the following. Consider a constant cocycle given by 
\[
\begin{bmatrix}
1 & 1 \\
0 & 1
\end{bmatrix}.
\]
We have written this matrix with respect to the usual coordinate framing $[e_1,e_2]$ of $\R^2$. If this is a cocycle over a shift, $\sigma\colon \Sigma\to \Sigma$, then for any measurable function $\phi\colon \Sigma\to \R$, we may consider this cocycle with respect to the framing $[e_1,\phi(x)e_1+e_2]$. With respect to this framing, the cocycle is:
\[
\begin{bmatrix}
1 & 1-\phi(\sigma(x))+\phi(x)\\
0 & 1
\end{bmatrix}.
\]
This new cocycle also takes values in a unipotent subgroup. However, this example shows that even in this simple case we cannot hope that a measurable framing would coincide with a continuous framing unless other requirements are satisfied. Thus we cannot have that $u'=u$ a.e.~in the conclusion of Theorem \ref{thm:Zimmer_block_rigidity}.

\noindent{\textbf{Acknowledgements}:} The author is grateful to Boris Kalinin and Daniel Mitsutani their helpful comments on an earlier version of this paper. The author is also grateful to Jairo Bochi, Aaron Brown, Amie Wilkinson, and Federico Rodriguez Hertz for helpful discussions.

\section{Setting}

Fix some number $\ell\ge 2$, which will be the number of symbols used to define the shift. Let $Q=(q_{i,j})_{1\le i,j\le \ell}$ be an irreducible matrix with each $q_{ij}\in \{0,1\}$. Let $\Sigma$ be the associated two-sided subshift of finite type. Namely, 
\[
\Sigma\coloneqq \{(x_n)_{n\in \Z}: q_{x_n,x_{n+1}}=1\text{ for all } n\in \Z\}\subset \{1,\ldots,\ell\}^\Z.
\]
Let $\sigma\colon \Sigma\to \Sigma$ be the left shift, i.e. $\sigma((x_n)_{n\in \Z})=(x_{n+1})_{n\in \Z}$. Also, for each $\tau>0$, we naturally equip $\Sigma$ with the metric 
\[
\rho_{\tau}(x,y)=e^{-\tau N(x,y)},
\]
where 
\[
N(x,y)=\max\{N\ge 0; x_n=y_n\text{ for all } \abs{n}<N\}.
\]

\begin{defn}
Let $A\colon \Sigma\to \GL(d,\R)$ be a Borel measurable map. The \emph{linear cocycle} over $f$ generated by $A$ is the map
\begin{align*}
\mc{A}&\colon \Sigma\times \Z\to \GL(d,\R),\\
\mc{A}(x,n)&=\begin{cases}
A(\sigma^{n-1}(x))\cdots A(\sigma(x))A(x)& n>0\\
\Id & n=0\\
A(\sigma^{-\abs{n}}(x))^{-1}\cdots A(\sigma^{-1}(x))^{-1}& n<0.
\end{cases}
\end{align*}
\end{defn}

We say that $A$ generates the associated cocycle $\mc{A}$. 
We write $A^n(x)$ for $\mc{A}(x,n)$. 
We say that a cocycle $\mc{A}$ is continuous if it has a continuous generator $A$.
Similarly we define $\mc{A}$ to be $\alpha$-H\"older continuous if it is generated by an $\alpha$-H\"older continuous $A$ with respect to the metric $\rho_{\tau}$ for some choice of $\tau>0$. We say that $\mc{A}$ is H\"older if this is true for some $\alpha>0$ and $\tau>0$. In a standard abuse, we also refer to the generator $A$ as a ``cocycle."

We now define the notions of cohomologous cocycles used in this paper.

\begin{defn}
Let $\Sigma$ be a subshift of finite type endowed with an invariant probability $\mu$. We say that two cocycles $A,B\colon \Sigma\to \GL(d,\R)$ are $\mu$-measurably cohomologous if there is a $\mu$-measurable function $P\colon \Sigma\to \GL(d,\R)$ such that 
\[
A(x)=P(\sigma(x))B(x)P(x)^{-1},
\]
for $\mu$-a.e. $x\in \Sigma$.
\end{defn}

\begin{defn}
Let $\Sigma$ be a subshift of finite type. We say that two cocycles $A,B\colon \Sigma\to \GL(d,\R)$ are continuously cohomologous if there is a continuous function $P\colon \Sigma\to \GL(d,\R)$ such that 
\[
A(x)=P(\sigma(x))B(x)P(x)^{-1},
\]
for all $x\in \Sigma$. Similarly, we say that $A$ and $B$ are H\"older cohomologous if $P$ may be taken to be H\"older continuous.
\end{defn}

We now develop the language to define local product structure for a measure. 
For $x\in \Sigma$, we define the local stable set of $x$ to be
\[
W^s_{loc}(x)\coloneqq \{(y_n)_{n\in \Z}\in \Sigma: x_n=y_n\text{ for all } n\ge 0\},
\]
and the \emph{local unstable} set by
\[
W^u_{loc}(x)\coloneqq \{(y_n)_{n\in \Z}\in \Sigma: x_n=y_n\text{ for all } n\le 0\}.
\]

For two points $x,y$ with $x_0=y_0$, we define $[x,y]$ to be the unique point in $W^u_{loc}(x)\cap W^s_{loc}(y)$, which is the point
\[
(\ldots,x_{-2},x_{-1},y_0,y_1,y_2,\ldots).
\]

Let
\begin{align*}
\Sigma^u&=\{(x_n)_{n\ge 0}: q_{x_nx_{n+1}}=1 \text{ for all } n\ge 0\}, \text{ and}\\
\Sigma^s&=\{(x_n)_{n<0}: q_{x_nx_{n+1}}=1\text{ for all } n\le -2\}.
\end{align*}

We denote by $\pi_u$ and $\pi_s$ the obvious projections of $\Sigma$ onto $\Sigma^u$ and $\Sigma^s$, respectively. We also use the usual notation for cylinders:
\[
[m; a_0,a_1,\ldots, a_k]=\{x\in \Sigma: x_{m+i}=a_{m+i}, 0\le i\le k\}.
\]

Let $\mu$ be a fully supported ergodic $\sigma$-invariant probability measure on $\Sigma$. 
We let $\mu^u=(\pi_u)_*\mu$ and $\mu^s=(\pi_s)_*\mu$. We denote by $\mu\vert E$ the restriction of $\mu$ to a set $E$. 

For each $1\le i \le \ell$, we have a natural map $[0;i]\to \pi_s([0;i])\times \pi_u([0;i])$ given by $\pi_s\times \pi_u$. We say that $\mu$ has \emph{continuous local product structure} if there is a continuous function $\xi\colon \Sigma\to (0,\infty)$ such that for each $i$:
\[
\mu\vert_{[0;i]}=\xi\cdot (\mu^s\vert\pi_s([0;i])\times \mu^u\vert \pi_u([0;i])).
\]
It is helpful to note that a measure with local product structure is locally ergodic. Hence as the shifts we consider are transitive, we see that the measures we consider are ergodic.

For an ergodic invariant measure $\mu$, we define the \emph{extremal Lyapunov exponents} of $\mc{A}$ with respect to $\mu$ to be
\begin{align*}
\lambda_{+}(\mu)&\coloneqq \inf_{n\ge 1}\frac{1}{n}\int_{\Sigma} \log \|A^n\|\,d\mu,\\
\lambda_{-}(\mu)&\coloneqq \sup_{n\ge 1}\frac{1}{n}\int_{\Sigma}\log \|A^{-n}\|^{-1}\,d\mu.
\end{align*}
We denote the extremal Lyapunov exponents of the uniform  measure supported on the orbit of a periodic point $p$ by $\lambda_+(p)$ and $\lambda_{-}(p)$. A useful fact will be the following.

\begin{prop}\label{prop:structure_implies_zero_exponents}
Let $\Sigma$ be a subshift of finite type and let $\mu$ be an ergodic measure on $\Sigma$. 
Suppose there is a measurable function $\alpha\colon \Sigma\to \SL(d,\R)$ such that after conjugating by $\alpha$, $A$ takes values in a Zimmer block with exponent $0$. 
Then $\lambda_+(\mu)=-\lambda_{-}(\mu)=0$.
\end{prop}
The proof of Proposition \ref{prop:structure_implies_zero_exponents} is similar to the proof of \cite[Prop. 4.2]{butler2018measurable} and is based on studying the returns to a Lusin set.

\begin{proof}
We give the proof in the case that the block is the standard upper triangular unipotent subgroup of $\SL(2,\R)$. Let $[e_1,e_2]$ denote the measurable framing that gives the cocycle its block structure. Fix a Lusin set $K$ for this framing so that $\mu(K)=1-\epsilon$. Pick a Lyapunov regular point $x\in K$. Let $\sigma_1,\tau_1,\sigma_2,\tau_2,\ldots$ be the sequence of sojourn times of $x$ in $K$, then in $\Sigma\setminus K$, then in $K$, and so on. We assume that $x$ equidistributes between $K$ and $\Sigma\setminus K$. Note that on $K$ there is a uniformly continuous function $a(x)$ such that with respect to the framing $[e_1,e_2]$ we have 
\[
A(x)=\begin{bmatrix}
1 & a(x)\\
0 & 1.
\end{bmatrix}
\]
In particular we see that 
\[
A^{\sigma_1}(x)e_1=\sum_{i=0}^{\sigma_1-1}a(\sigma^i(x))e_1+e_2.
\]
Note that there exists an increasing polynomial $p(n)$ such that for any $y\in K$, if $y$ stays in $K$ for $n$ iterates $\abs{\sum_{i=0}^{n-1}a(\sigma^i(y))}\le p(n)$.

After $\sigma_1$ iterates, $x$ leaves $K$. As the measurable block structure extends outside of $K$, we see that $A^{\tau_1}(\sigma^{\sigma_1}(x))e_1=e_1$, and we can estimate the norm of $A^{\tau_1}(\sigma^{\sigma_1}(x))$ by using merely that the cocycle is continuous, i.e. $\|A\|\le e^{\beta}$ for some $\beta>0$. Thus we see that for any $y$,
\[
\|A^{\tau_1}(y)e_1\|\le e^{\beta \tau_1}.
\]
So, if we write $A^{\sigma_1+\tau_1}(x)e_2=c(1)e_1+d(1)e_2$, then $\abs{c(1)}\le p(\sigma_1)e^{\tau_1\beta}$. If we let $c(n),d(n)$ be the corresponding coefficients for $A^{\sum_{i=1}^n \sigma_i+\tau_i}(x)e_2$, then by the upper triangular structure $d(n)=1$ and 
\[
\abs{c(n)}\le \sum_{i=1}^n p(\sigma_i)e^{\beta(\tau_i+\cdots+\tau_n)}.
\]
If we let $N=\sum_{i=1}^n \sigma_i+\tau_i$, then the above line is bounded by $Np(N)e^{\beta(\tau_1+\cdots+\tau_n)}$.

But we can now estimate the Lyapunov exponent associated to the vector $e_2$, which needs to equal $\lambda_+(\mu)>0$ if the cocycle has a non-zero Lyapunov exponent. Now,
\[
\lambda_+(\mu)=\limsup_{N\to \infty}\frac{1}{N}\log \|A^N(x)e_2\|\le C\beta\epsilon.
\]
But this holds for each $\epsilon$ as we can always choose a larger Lusin set. Thus we are done.

A more complicated variant on this idea works in higher dimensions.
\end{proof}

\section{Preliminaries on Cocycles}\label{sec:preliminaries}

In the course of the proof, we will work with certain uniformity sets for the growth of the angle distortion of the cocycle. Let $\mc{D}(N,\theta)$ denote the set of points $x\in \Sigma$ such that for all $s\ge 1$,
\[
\prod_{j=0}^{s-1} \|A^N(f^{jN}x))\|\|A^N(f^{jN}(x))^{-1}\|\le e^{sN\theta} 
\]
and
\[
\prod_{j=0}^{s-1}\|A^{-N}(f^{-jN}(x))\|\|A^{-N}(f^{-jN}(x))^{-1}\|\le e^{sN\theta}.
\]
These sets are useful because they measurably exhaust $\Sigma$.

\begin{lem}\label{lem:blocks_exhaust}
\cite[Cor. 2.4]{viana2008almost} Suppose that $\lambda^+(\mu)=\lambda^-(\mu)=0$. Then $\mu$-a.e.~$x\in \Sigma$ is in $\mc{D}(N,\theta)$ for some $N>0$ and $0<\theta<\tau$.
\end{lem}
An important property of the sets $\mc{D}(N,\theta)$ is that the holonomies are uniformly continuous on them.
\begin{prop}\label{prop:viana}
\cite[Prop.~2.5, Cor.~2.8]{viana2008almost} Given $N,\theta$ with $\theta<\tau$, there exists $L=L(N,\theta)>0$ such that for any $x\in \mc{D}(N,\theta)$ and any $y,z\in W^s_{loc}(x)$, 
\[
H^s_{yz}=\lim_{n\to \infty} A^n(z)^{-1}A^n(y)
\]
exists and satisfies $\|H^s_{yz}-\Id\|\le L\cdot \rho(y,z)$ and $H^s_{yz}=H^s_{xz}\circ H^s_{yx}$. Further, $H_{xy}^s=A^{-n}(f^n(y))H_{f^n(x)f^n(y)}^sA^n(x)$. The analogous statement holds for unstable holonomies. 
\end{prop}
In fact, one even has continuity of the holonomies transversely because they converge uniformly. Although it is not mentioned in Butler's paper, this does seem to used there. This is a standard fact, but the only proof we are aware of is \cite[Prop.~2.10]{bochi2019extremal}, which is written under the assumption of fiber bunching. As our cocycle is not \emph{a priori} fiber bunched, we include our own proof for the sake of completeness.

\begin{lem}
For each $N,\theta$, there exist $C,\gamma>0$ such that if $x,x'\in \mc{D}(N,\theta)$ and $x'\in W^s_{loc}(x)$, $y'\in W^s_{loc}(y)$ and $x'\in W^u_{loc}(y')$, then
\[
\|H^s_{xy}-H^s_{x'y'}\|\le Cd(x,x')^{\gamma}.
\]
The analogous statement also holds for the unstable holonomies.
\end{lem}
\begin{proof}
We will obtain an estimate under the assumption that $d(x,x')\le e^{-n\tau'}$ where $\tau'\ge 2\tau+2\eta+\ln(2)$, where $e^{\eta}\ge\max_{x\in M}\{ \|A^N(x)\|,\|A^{-N}(x)\|\}$. In the rest of the proof, we will assume that $N=1$ to simplify the notation.

To begin, we have that 
\[
H_{xy}^s=A^{-n}(f^n(y))H_{f^n(x),f^n(y)}^sA^n(x).
\]
By the Lipschitz estimate on the holonomies in Proposition \ref{prop:viana}, we have that 
\[
H_{xy}^s=\Id+O(d(x,y)),
\]
where the big-O is independent of the local stable leaf $x$ and $y$ lie in. Hence combining the two previous equations:
\begin{equation}\label{eqn:h_xyexpanded}
H_{xy}^s=A^{-n}(f^n(y))A^n(x)+A^{-n}(f^n(y))O(d(f^n(y),f^n(x)))A^n(x).
\end{equation}
By \cite[Lem.~2.6]{viana2008almost}, there exists a $C>0$ independent of $x$ and $y$ such that $\|A^{-n}(f^n(y))\|\|A^n(x)\|<Ce^{n\theta}$. Hence we see that the last term in equation \eqref{eqn:h_xyexpanded} is $O(e^{n(\theta-\tau)})$.

To conclude it remains to estimate 
\[
A^{-n}(f^n(y))A^n(x)-A^{-n}(f^n(y'))A^n(x')=H^s_{xy}-H^s_{x'y'}+O(e^{n(\theta-\tau)}).
\]
For this we first estimate 
\[
A^n(x)-A^n(x')
\]
by using that $x$ and $x'$ are assumed to be close. For this write
\[
A^n(x')=\prod_{i=1}^n \left(A(f^i(x))+O(e^{-n\tau'+i\tau})\right),
\]
but by our choice of $\tau'\ge 2\tau+2\eta+\ln(2)$, we see that
\[
A^n(x)=A^n(x')+O(e^{-\tau n-\eta n}).
\]
We now estimate 

\begin{align*}
&A^{-n}(f^n(y))A^n(x)-A^{-n}(f^n(y'))A^n(x')\\
&=A^{-n}(f^n(y))(A^n(x)-A^n(x'))+(A^{-n}(f^n(y))-A^{-n}(f^n(y')))A^n(x')\\
&=O(e^{n\eta})O(e^{(-\tau-\eta)n})+O(e^{(-\tau-\eta)n})O(e^{n\eta})\\
&= O(e^{-\tau n}).
\end{align*}
Hence we see that 
\[
\|H_{xy}^s-H_{x'y'}^s\|=O(e^{n(\theta-\tau)}).
\]
We assumed at the beginning that $d(x,x')\le e^{-n\tau'}$. Hence if $(\theta-\tau)=\gamma\tau'$ for some $0<\gamma<1$, we see that 
\[
\|H_{xy}^s-H_{x'y'}^s\|\le Cd(x,x')^{\gamma},
\]
as desired.
\end{proof}

We already know that the holonomies are H\"older continuous as we vary $x$ and $y$ within a leaf, hence we may strengthen Proposition \ref{prop:viana} slightly.
\begin{prop}\label{prop:transverse_holonomy_continuity}
In the context of Prop. \ref{prop:viana}, the stable holonomy $H^s_{xy}$ depends continuously on the points $x,y\in \mc{D}(N,\theta)$ (even when $x$ and $y$ vary transversely to the common local stable leaf they lie in).
\end{prop}

It is useful to note that if we have a holonomy $H^s_{xy}\colon \R^d\to \R^d$, then this holonomy induces holonomies on other spaces such as $\Gr(k,d)$. It is useful to note these induced holonomies may be defined even when bunching conditions do not hold for the corresponding induced cocycles. For more discussion about cocycles and related constructions, see \cite{avila2013holonomy}.

The following lemma will be used several times during the proof.

\begin{lem}\label{lem:piece_together_lemma}
Suppose that $A$ is a cocycle as above.
If $\Sigma\times \R^n$ has a continuous invariant subbundle $V_1$ such that $V_1$ admits a continuous framing so that with respect to the framing of $V_1$ $A$ takes values in $\lambda K$ for some compact group $K$, and $W/V_1$ has a framing so that the quotient cocycle $\overline{A}\colon W/V_1\to W/V_1$ takes values in a Zimmer block $U_{\lambda}$, then there exists a continuous framing of $\Sigma\times \R^n$ such that, with respect to this framing, $A$ takes values in a Zimmer block $U_{\lambda}'$. This is also true, mutatis mutandis, for framings not defined over all of $\Sigma$ and for H\"older continuity.
\end{lem}

\begin{proof}
Let $[e_1,\ldots,e_k]$ be the framing of $V_1$. Let $[e_{k+1},\ldots,e_{n}]$ be the framing of $W/V_1$. Fix a continuous complement $V'$ to $V_1$ inside of $W$. Then for $k+1\le i\le n$, let $\wt{e}_i$ be a lift of $e_i$ to $\R^n$ with respect to this complement. With respect to the continuous framing $[e_1,\ldots,e_k,\wt{e}_{k+1},\ldots,\wt{e}_{n}]$ the cocycle lies in such a Zimmer block.
\end{proof}

\noindent \textbf{Warning:} In what follows, we will assume that $\lambda=0$. We can always assume this by replacing the cocycle $A$ with the cocycle $e^{-\lambda}A$.

\section{Sketch of Proof}
The steps in the proof of Theorem \ref{thm:Zimmer_block_rigidity} follow \cite{butler2018measurable}.

\begin{enumerate}
\item Given that $\lambda_+(\mu)=0$, $\Sigma$ is measurably exhausted by the sets $\mc{D}(N,\theta)$ on which the holonomies are uniformly continuous.

\item We show that over the sets $\mc{D}(N,\theta)$ that the measurable block structure gives a uniformly continuous flag as well as metrics on the flag's quotients, so that with respect to these structures the cocycle looks like a Zimmer block.

\item We show that there exists a periodic point with $\lambda_+(p)=\lambda_{-}(p)=0$.

\item If there is a periodic point with $\lambda_+(p)-\lambda_{-}(p)>0$, we then use a shadowing lemma of Butler to produce new periodic points whose return map has large norm yet whose orbits lie in $\mc{D}(N,\theta)$; through a shadowing argument we show that this is incompatible with the block structure.

\item Once we know that all periodic points have $\lambda_+(p)=\lambda_{-}(p)=0$, then this implies that for some $N,\theta>0$, $\mc{D}(N,\theta)$ contains all of $\Sigma$ and thus we are done by the second step.

\end{enumerate}

\section{Uniform Continuity of Dynamical Structures}

In this section we give a number of results concerning holonomies and the sets $\mc{D}(N,\theta)$ that we use in the proof of Theorem \ref{thm:Zimmer_block_rigidity}.

At almost every $x\in \Sigma$, the measurable transfer function $u$ gives a measurable invariant flag 
\[\{0\}=\mc{E}_0\subset \mc{E}_1\subset \cdots \subset \mc{E}_k=\R^d,\]
 where $k$ is equal to the number of compact diagonal blocks in $U_{\lambda}$. 
The measurable framing also gives a measurable metric structure by declaring the framing to be orthonormal.  By quotienting, we obtain such a metric and framing of $\mc{E}_{i+1}/\mc{E}_i$ for each $i$.

Our first goal is to show that these bundles and the associated metric are uniformly continuous and holonomy invariant on the sets $\mc{D}(N,\theta)$.
 It suffices to check these things for just $\mc{E}_1$, as Lemma \ref{lem:piece_together_lemma} allows us to obtain the result for the whole flag by quotienting. 
Lemma \ref{lem:inv_on_full_measure} establishes the holonomy invariance of these structures on a full measure set.

Before we prove Lemma \ref{lem:inv_on_full_measure}, we recall a useful invariance principle due to Ledrappier, which requires some additional language to state. Let $(M,\mc{M},\mu)$ be a Lebesgue measure space. Let $T\colon M\to M$ be an invertible measurable transformation. A $\sigma$-algebra $\mc{B}\subset\mc{M}$ is \emph{generating} if the algebras $T^n(\mc{B})$, $n\in \Z$, generate $\mc{M}$ mod $0$, i.e. for every $E\in \mc{M}$, there exists $E'$ in the smallest $\sigma$-algebra that contains all the $T^n(\mc{B})$ such that $\mu(E\Delta E')=0$. 

\begin{lem}\label{lem:ledrappier}
\cite{ledrappier1986positivity} Let $(M,\mc{M},\mu)$ be a Lebesgue space. Let $B\colon M\to \GL(d,\R)$ be a measurable cocycle such that the functions $x\mapsto \log \|B(x)^{\pm}\|$ are $\mu$-integrable. 
Let $\mc{B}\subset \mc{M}$ be a generating $\sigma$-algebra such that both $T$ and $B$ are $\mc{B}$-measurable $\mod 0$. 
If $\lambda_-(x)=\lambda_+(x)$ for $\mu$-a.e.~$x\in M$, then for any invariant probability measure $m$ on $\Sigma\times \mathbb{P}(\R^d)$ that projects to $\mu$, any disintegration $x\mapsto m_x$ of $m$ along the projective space fibers is $\mc{B}$-measurable mod $0$.
\end{lem}

The following lemma shows that there exists a full measure set where the measurable flag has holonomy invariance.

\begin{lem}\label{lem:inv_on_full_measure}
There is a $\sigma$-invariant subset $\Omega$ of $\Sigma$ with $\mu(\Omega)=1$ such that every $x\in \Omega$ is in $\mc{D}(N_x,\theta_x)$, with $\theta_x<\tau$, and such that if $x,y\in \Omega$ and $y\in W^*_{loc}(x)$ then if $\eta$ denotes $\mc{E}_1,\ldots, \mc{E}_k$, or the metrics on $\mc{E}_{i+1}/\mc{E}_i$, then
\[
H^*_{xy}\eta_x=\eta_y,
\]
where $*\in \{s,u\}$.
Further, for every $n\in \Z$ we have that 
\[
A^n(x)[\eta_x]=\eta_{\sigma^n(x)}.
\]
\end{lem}

\begin{proof}

We will show that the result holds when $\eta$ denotes the bundle $\mc{E}_1$. 
The invariance of the other bundles then follows similarly.
The result about the invariance of the metrics has a different proof but is identical to \cite[Lem. 4.4]{butler2018measurable} and hence is omitted.

Our proof will use Lemma \ref{lem:ledrappier}, which requires some setup for its application. 
To begin, by Proposition \ref{lem:blocks_exhaust} for a.e.~$x\in \Sigma$, there exist $N>0$ and $0<\theta<\tau$ such that $x\in \mc{D}(N,\theta)$. Thus by Lemma \ref{prop:viana} for a.e.~$x\in \Sigma$, there are holonomies defined between the points in $W^s_{loc}(x)$. Let $\mc{D}$ denote the union of the sets $\mc{D}(N,\theta)$ over all $N,\theta>0$. Note that $\mu(\mc{D})=1$.

Let $r=\dim \mc{E}_1$ and consider the associated Grassmannian bundle $\Sigma\times \Gr_r(\R^d)$. The cocycle $A$ also induces a cocycle $\wt{A}$ on this Grassmannian bundle. Further, the stable and unstable holonomies $H^*$ induce holonomies for $\Sigma\times \Gr_r(\R^d)$ which we also denote by $H^*_{xy}$.

To begin, we will measurably conjugate the cocycle so that it is constant along local stable manifolds. This is a standard construction: see, for example, \cite[Cor. 1.15]{bonatti2003genericite}. To this end, we define a map $h\colon \Sigma\times \R^n\to \Sigma\times \R^n$.  By the product structure, we can choose points $\omega^i\in [0;i]$ such that for almost every $y\in [0;i]$, $W^u_{loc}(\omega^i)\cap W^s_{loc}(y)$ is a point in $\mc{D}$. We now define $h$ by
\[
h\colon (y,\xi)\mapsto (y,H_{[\omega^i,y]y}^s(\xi)).
\]
By our choice of the basepoints, $h$ is measurable.

We denote the conjugated cocycle by $\mc{A}_s$. 
By the final relationship in Proposition \ref{prop:viana}
, $\mc{A}_s$ is almost surely constant on local stable leaves, i.e. for almost every $y$,
\[
\mc{A}_s\colon (y,\xi)\mapsto (\sigma(y),A([\omega_i,y])\xi).
\]
In particular as $A$ is continuous, this shows that the integrability assumption of Lemma \ref{lem:ledrappier} is satisfied.

The measurable invariant subbundle $\mc{E}_1$ gives a measurable invariant section of $\Sigma\times \Gr_r(\R^d)$, which we denote by $\eta$, as before. 
Let $m$ be the lift of $\mu$ to the graph of $\eta$ inside of $\Sigma\times \Gr_r(\R^d)$. 
This is the same as defining $m$ to be the measure projecting to $\mu$ such that for $\mu$-a.e.~$x\in \Sigma$, the conditional measure $m_x$ on the fiber $\{x\}\times \Gr_r(\R^d)$ is the atomic measure at $[\mc{E}_1(x)]\in \Gr_r(\mathbb{R}^d)$. Since $\eta$ is invariant, so is the measure $m$.

Let $\mc{B}$ be the partition of $\Sigma$ into local stable manifolds. The shift $\sigma$ is $\mc{B}$ measurable. Further as $\mc{A}_s$ is constant on local stable manifolds, we see that $\mc{A}_s$ is $\mc{B}$ measurable mod $0$. Thus by Lemma \ref{lem:ledrappier}, we see that the disintegration of $m$, $x\mapsto m_x$ is $\mc{B}$ measurable modulo $0$. 
Due to the product structure of $\mu$, this implies $m_x$ is almost surely constant on almost every local stable manifold. 
But for a.e.~$x$, $m_x$ is a Dirac mass on the image of $\eta_x$ under the conjugacy, $h$. 
Because $\mc{A}_s$ is a.s.~constant on local stable leaves, its stable holonomies are trivial.
Thus as $m_x$ is almost surely constant on almost every local stable leaf, it is invariant under the stable holonomies almost everywhere.

As is easily checked, for points $x,y$ where $h$ is defined and the $\mc{A}_s$ and $\mc{A}$ holonomies between $x$ and $y$ exist, $h$ intertwines these holonomies. 
Thus as the location of the Dirac mass $m_x$ is the image of $\eta_x$, we see that $\eta$ is almost surely invariant under the stable holonomies arising from $\mc{A}$.  
This means that there exists a full $\mu$-measure set $S\subset \Sigma$ such that if $x,y\in S$ and $y\in W^s_{loc}(x)$, then
\[
H_{xy}^s\eta_x=\eta_y,
\]
which is the desired conclusion.
The same argument gives the analogous result for the unstable holonomies. The claim about the invariance of $\eta$ under the cocycle is immediate.
\end{proof}

Our goal is to extend these flags and metrics to be defined on all of $\Sigma$. To achieve this, we will use some results of Butler concerning the sets $\mc{D}(N,\theta)$.

\begin{prop}\label{prop:butler45}
\cite[Lemma 4.5]{butler2018measurable} Suppose that we have a cocycle with $\lambda_+(\mu)=\lambda_-(\mu)=0$. Then for each $N>0$ and $\theta<\tau$, there is an $N_*\ge N$ and $\theta\le \theta_*<\tau$ such that 
\[
\mc{D}(N,\theta)\subseteq \supp(\mu\vert \mc{D}(N_*,\theta_*)).
\]
\end{prop}

Before passing to the next result we need, let us comment a little on the proof of Proposition \ref{prop:butler45}, which is a delicate shadowing argument. In order to produce the points in $\mc{D}(N_*,\theta_*)$, one finds trajectories that begin very close to $\mc{D}(N,\theta)$ and that enter a set on which quasi-conformal distortion is small by time they have separated from $\mc{D}(N,\theta)$. As the trajectories Butler constructs stay close to sets where there is little quasi-conformal distortion, one can conclude that they lie in $\mc{D}(N_*,\theta_*)$ for some larger $N_*$ and $\theta_*$.

We will also use the following property.
\begin{prop}\label{prop:butler47}
\cite[Prop. 4.7]{butler2018measurable} Let $\epsilon>0$ be given. Then for every $N$ large enough we have for all $\theta>0$,
\[
\sigma(\mc{D}(N,\theta))\subseteq \mc{D}(N,\theta+\epsilon).
\]
\end{prop}

We now prepare to prove Proposition \ref{prop:cts_on_pesin_set}, which essentially says that over a uniformity set $\mc{D}(N,\theta)$ the measurable structures preserved by the cocycle coincide with uniformly continuous structures.
The idea of the proof is to use that $\mc{D}(N,\theta)$ is contained in $\supp(\mu\vert \mc{D}(N_*,\theta_*))$ for some $N_*,\theta_*$.
There exists a full measure set $\Omega$ where the holonomies are defined, hence as holonomies are uniformly continuous on $\mc{D}(N_*,\theta_*)$ by Proposition \ref{prop:transverse_holonomy_continuity}, they give a well defined extension to $\mc{D}(N,\theta)$. The following is based on \cite[Lem. 4.6]{butler2018measurable}. We have included a full proof and restatement because the original statement and proof contain some minor oversights. For example, in the statement of \cite[Lem.~4.6]{butler2018measurable}, if $\sigma(x)\notin \mc{D}(N,\theta)$, then $\hat{\eta}_{\sigma(x)}$ need not be defined.

\begin{defn}\label{def:holonomy_invariant}
(Holonomy Invariance) For a set $X\subseteq \bigcup_{N,\theta<\tau}\mc{D}(N,\theta)$, we say that a choice of subspace $\eta\colon X\to \Gr_r(\R^n)$ is \emph{holonomy invariant} if for each $x,y\in X$, we have $H^u_{[y,x]y}H^s_{x[y,x]}\eta_x=\eta_y$ and $H^s_{[x,y]y}H^u_{x[x,y]}\eta_x=\eta_y$. (Note that these holonomies are defined because we assume $x$ and $y$ each live in some regularity block $\mc{D}(N,\theta)$, $\theta<\tau$.) We write $H^{us}_{xy}$ and $H^{su}_{xy}$ for each of these composed holonomies. We analogously speak of the holonomy invariance of any other structure.
\end{defn}

\begin{rem}\label{rem:extensions}
An important property of a set $X\subseteq \mc{D}(N,\theta)$ on which $\eta$ is holonomy invariant is that the continuous extension of $\eta$ to $\overline{X}$ is also holonomy invariant because $\overline{X}\subseteq \mc{D}(N,\theta)$ and the transverse uniform continuity of the holonomies from Proposition \ref{prop:transverse_holonomy_continuity}.
\end{rem}

Before we continue, we will show that the full measure set $\Omega$ may be taken to have the subspaces and metrics $\eta$ holonomy invariant in sense of Definition \ref{def:holonomy_invariant}.

\begin{lem}\label{lem:holonomy_invariant}
In Lemma \ref{lem:inv_on_full_measure}, the set $\Omega$ may be taken to be holonomy invariant.
\end{lem}

\begin{proof}
What we need to show is that if $x,y$ are two points in $\Omega$, then $H^{us}_{xy}\eta_x=H^{su}_{xy}\eta_x=\eta_y$. We may always assume that $x\in \mc{D}(N_x,\theta_x)$ and similarly for $y$, as the union of all the regularity blocks for the cocycle has full measure. We will only show that $H^{su}_{xy}\eta_x=\eta_y$ as the other case is similar.

First we establish a notion of a good point in $\Omega$. An $su$-good point is a point $x\in \Omega$ such that $W^s_{loc}(x)\cap \Omega$ has full measure with respect to the conditional measure of $\mu$ on $W^s_{loc}(x)$, and such that for almost every point $y\in W^s_{loc}(x)$ with respect to the conditional measure on $W^u_{loc}(y)$ the set $\Omega\cap W^u_{loc}(y)$ has full measure. Note that these conditional measures admit a straightforward description due to the continuous product structure. If we did the saturation in the other order then we would call the point $us$-good. We say a point is \emph{good} if it is $su$- and $us$-good. We denote by $\mc{G}$ the set of all good points. Note that $\mu(\mc{G})=1$.

We claim that if $x,y\in \mc{G}\subseteq \Omega$, then $H^{su}_{xy}\eta_x=\eta_y$. To begin, we have from Proposition \ref{prop:butler45} that there exist $N_*,\theta_*$, such that $\mc{D}(N_y,\theta_y)\subseteq \supp(\mu\vert \mc{D}(N_*,\theta_*))$. In particular, as $y$ is in the support of $\mu$ this implies that there exists a sequence $A_n$ of neighborhoods of $y$ with radius going to $0$ such that $\mu(A_n)>0$. Because of the product structure of $\mu$, and because $x$ and $y$ are both good points, almost every point $q\in A_n\cap \Omega$ satisfies that $[y,q]$ and $[q,x]$ are in $\Omega$. Hence we may now pick a sequence of such points $q_n$ such that $q_n$ coverges to $y$. 
Then by the holonomy invariance of $\eta$ over $\Omega$ along local stable and unstable manifolds, we see that
\begin{align*}
H^u_{[q_n,x]q_n}H^s_{x[q_n,x]}\eta_x&=\eta_{q_n},\\
H^s_{[q_n,y]y}H^u_{q_n[q_n,y]}\eta_{q_n}&=\eta_{y}.\\
\end{align*}
Hence 
\[
H^s_{[q_n,y]y}H^u_{q_n[q_n,y]}H^u_{[q_n,x]q_n}H^s_{x[q_n,x]}\eta_x=\eta_y.
\]
Because each $q_n\in \mc{D}(N_*,\theta_*)$, these holonomies are all uniformly continuous, hence as $n\to \infty$ the previous line converges to:
\[
H^u_{[y,x]y}H^s_{x[y,x]}\eta_x=\eta_y,
\]
which is exactly what we desired. That $H^{us}_{xy}\eta_x=\eta_y$ follows similarly.

Thus we have shown that $\mc{G}\subseteq \Omega$ has the required property. To ensure that $\mc{G}$ is $\sigma$-invariant, we may replace it with $\cap_{n\in \Z} \sigma^n(\mc{G})$.
\end{proof}

The following proposition shows that we can construct a continuous version of $\eta$ over a block $\mc{D}(N,\theta)$.

\begin{prop}\label{prop:cts_on_pesin_set}
Suppose that $\Sigma$ is a transitive subshift of finite type, $\mu$ is a fully supported measure with continuous local product structure, and $A\colon \Sigma\to \GL(d,\R)$ is a H\"older continuous cocycle such that $\lambda_+(\mu)=\lambda_{-}(\mu)=0$. Suppose that $\eta$ is a measurable $A$-invariant section of $\mathbb{P}(\R^d)$.

For each $\theta<\tau$ and $N>0$, there exists a closed set $S$ and a continuous function $\hat{\eta}\colon S\to \mathbb{P}(\R^d)$ such that the following hold:

\begin{enumerate}
\item
$\mc{D}(N,\theta)\cup \sigma(\mc{D}(N,\theta))\subset S$,
\item
$\hat{\eta}=\eta$ $\mu$-a.e.~on $\mc{D}(N,\theta)\cup \sigma(\mc{D}(N,\theta))$,
\item
$\hat{\eta}$ is holonomy invariant,
\item
For $x\in \mc{D}(N,\theta)$,
\[
A(x)\eta_x=\eta_{\sigma(x)}.
\]
\end{enumerate}
The same conclusions hold if $\eta$ had instead been a measurable invariant metric on a measurable invariant subbundle of $\R^d$, i.e., the subbundle has a uniformly continuous version over $S$ and $\eta$ agrees over $S$ with a uniformly continuous metric $\mu$-a.e.
\end{prop}

Before we begin the proof, we describe its approach. The idea is that on a set of the form $\mc{D}(N_*,\theta_*)$, the holonomies are uniformly continuous and the further restriction of $\eta$ to $\Omega\cap \mc{D}(N_*,\theta_*)$ is holonomy invariant. Hence we can find a uniformly continuous version of $\eta$ on $\supp(\mu\vert \mc{D}( N_*,\theta_*))=\supp(\mu\vert\Omega\cap \mc{D}(N_*,\theta_*))$. As for sufficiently large $N_*,\theta_*$ this set contains $\mc{D}(N,\theta)$, we may use this set to obtain the needed extension.  Obtaining the cocycle invariance of this uniformly continuous version of $\eta$ then requires an additional argument where we choose $N_*,\theta_*$ even larger so that this set also contains $\sigma(\mc{D}(N,\theta))$.

\begin{proof}

We will construct a continuous function $\hat{\eta}$ on a set $S$ by extending from a subset of $\Omega$ on which the holonomies are uniformly continuous. To do this we will consider some nested sets $\mc{D}(N,\theta)$.

According to Lemmas \ref{lem:inv_on_full_measure} and \ref{lem:holonomy_invariant}, there exists a full measure subset $\Omega$ on which we already have holonomy invariance of $\eta$ as well as invariance under the cocycle.
To begin, $N,\theta$ are given. 
Apply Proposition \ref{prop:butler45} to obtain $N_*$ and $\theta_*$ such that
\[
\mc{D}(N,\theta)\subseteq \supp(\mu \vert \mc{D}(N_*,\theta_*)).
\]
By possibly enlarging $N_*$, and $\theta_*$, by Proposition \ref{prop:butler47} we can also ensure that 
\[
\mc{D}(N,\theta)\cup\sigma(\mc{D}(N,\theta))
\subset \supp(\mu\vert \mc{D}(N_*,\theta_*)).
\]
Now choose $0<\epsilon<\tau-\theta_*$. By Proposition \ref{prop:butler47}, and possibly replacing $N_*$ with a multiple of itself,
we may assume that
\begin{equation}\label{eqn:sigma_inv_epsilon}
\sigma(\mc{D}(N_*,\theta_*))\subset \mc{D}(N_*,\theta_*+\epsilon).
\end{equation}

We now construct the extension $\hat{\eta}$. Over the set $\Omega\cap \mc{D}(N_*,\theta_*+\epsilon)$, the holonomies are uniformly continuous, hence there exists a unique continuous extension of $\eta$ to the set $\overline{\Omega\cap \mc{D}(N_*,\theta_*+\epsilon)}$. We call this extension $\hat{\eta}$ and let $S=\overline{\Omega\cap \mc{D}(N_*,\theta_*+\epsilon)}$.

We now verify that $\hat{\eta}$ has the four properties required by the lemma:

\begin{enumerate}
\item[(1)]
The first property follows because $\overline{\Omega\cap \mc{D}(N_*,\theta_*+\epsilon)}$ contains both \\ $\supp(\mu\vert \mc{D}(N_*,\theta_*))$ and $\sigma(\supp(\mu\vert \mc{D}(N_*,\theta_*)))$ by equation \eqref{eqn:sigma_inv_epsilon}, and those two sets contain $\mc{D}(N,\theta)$ and $\sigma(\mc{D}(N,\theta))$, respectively.
\item[(2)]
The second item follows because on the set $\Omega\cap (\mc{D}(N,\theta)\cup \sigma(\mc{D}(N,\theta))$, $\hat{\eta}$ is equal to $\eta$ and $\Omega$ is a set of full measure.
\item[(3)]
That $\hat{\eta}$ is holonomy invariant (in the sense of Definition \ref{def:holonomy_invariant}), follows immediately from Remark \ref{rem:extensions}.
\item[(4)]
For the cocycle invariance of $\hat{\eta}$ over $\mc{D}(N,\theta)$, note by equation \eqref{eqn:sigma_inv_epsilon} we can define an auxiliary function $\zeta$ over the set $\supp(\mu\vert \mc{D}(N_*,\theta_*))$ by
\[
\zeta_x=A^{-1}(\sigma(x))\hat\eta_{\sigma(x)}.
\]
The function $\zeta$ is uniformly continuous. 
Now observe that on the set $\Omega\cap \mc{D}(N_*,\theta_*)$ that $\zeta_x=\eta_x=\hat\eta_x$ because $\eta$ is $A$-invariant over this set. 
Hence as $\Omega\cap \mc{D}(N_*,\theta_*)$ is dense in $\supp(\mu\vert \mc{D}(N_*,\theta_*))$, we see that $\zeta$ and $\hat{\eta}$ are equal on $\supp(\mu\vert\mc{D}(N_*,\theta_*))$ as both are continuous. 
Note that $A^{-1}(\sigma(x))\hat\eta_{\sigma(x)}$ is a continuous function defined on $\supp(\mu\vert\mc{D}(N_*,\theta_*))$ that agrees with $\zeta$ on this set. Hence for $x\in \supp(\mu\vert \mc{D}(N_*,\theta_*))$, 
\[
A^{-1}(\sigma(x))\hat\eta_{\sigma(x)}=\zeta_x=\hat\eta_x.
\]
As $\mc{D}(N,\theta)\subseteq \supp(\mu\vert \mc{D}(N_*,\theta_*))$, it follows for all $x\in \mc{D}(N,\theta)$ that
\[
A(x)\hat\eta_x=\hat{\eta}_{\sigma(x)}.
\]
\end{enumerate}

\end{proof}

From Lemma \ref{lem:piece_together_lemma} and Proposition \ref{prop:cts_on_pesin_set}, the following is immediate.
\begin{prop}
Suppose that $A,\Sigma$, and $\mu$ are as in Theorem \ref{thm:Zimmer_block_rigidity}. Then over $\mc{D}(N,\theta)$ there is a uniformly continuous invariant flag $\{\mc{E}_i\}$ that agrees with the measurable flag derived from the block structure almost everywhere. For each $0\le i<\ell$, there is a uniformly continuous metric on $\mc{E}_{i+1}/\mc{E}_i$ such that the quotient cocycle is isometric with respect to this metric.
\end{prop}

\section{Results about Periodic Orbits}\label{sec:periodic_orbits}

In this section, we begin to turn information about the measure $\mu$ into information about the periodic points of $\sigma$. 
We begin with the following, which is an analog of \cite[Lemma 4.8]{butler2018measurable} and has a similar proof.

\begin{lem}
Let $\Sigma,\mu$, and $A$ be as in Theorem \ref{thm:Zimmer_block_rigidity}.
 Then:
\begin{enumerate}
\item
There are no periodic points $p$ for which the inequality
\[
0< \lambda_+(p)-\lambda_{-}(p)<\tau
\]
holds.
\item
There is a periodic point $q$ such that $\lambda_+(q)=\lambda_{-}(q)=0$.
\end{enumerate}
\end{lem}
\begin{proof}[Proof Sketch.]
The main claim is that if $p$ is a periodic point with $0\le \lambda_+(p)-\lambda_-(p)<\tau$,
then the orbit of $p$ lies in $\mc{D}(N,\theta)$ for some $N$ and $\theta$.
For this see \cite[Lem. 4.8]{butler2018measurable}. 

If the entire orbit of $p$ lies in $\mc{D}(N,\theta)$, then by Proposition \ref{prop:cts_on_pesin_set}, over the orbit of $p$ the cocycle has a Zimmer block structure with constant $\lambda=0$.
 This implies that $\lambda_+(p)=\lambda_-(p)=0$.

To see that there is a periodic point with $\lambda_+(p)=\lambda_-(p)=0$, apply Kalinin's periodic approximation theorem \cite[Thm. 1.4]{kalinin2011livsic} to the measure $\mu$ to produce a periodic point with $\abs{\lambda_+(p)}+\abs{\lambda_-(p)}<\tau$. Such a periodic point has $\lambda_+(p)=\lambda_-(p)=0$ by the first part of the lemma.
\end{proof}

Our goal for finishing the proof is to show that the cocycle is fiber bunched, which can be deduced from the following lemma.

\begin{lem}\label{lem:pps_have_0_exponent}
Suppose that $A\colon \Sigma\to \GL(d,\R)$ is a cocycle measurably conjugated into a Zimmer block as in Theorem \ref{thm:Zimmer_block_rigidity}, then for every periodic point $p$, $\lambda_+(p)=\lambda_-(p)=0$.
\end{lem}

Lemma \ref{lem:pps_have_0_exponent} follows from an extension of the shadowing argument of Butler in \cite{butler2018measurable}. The idea of the extension is as follows and we illustrate it in the two dimensional case. Suppose we have two periodic points $x$, with $\lambda_+(x)=\lambda_-(x)=0$, and $y$, with $\lambda_+(y)-\lambda_-(y)> 0$. For simplicity, suppose both are fixed points. Butler constructs periodic points $p^m$ that first shadow $x$ and then $y$. These points are constructed so that $p^m\in \mc{D}(N,\theta)$ for some fixed $N$ and $\theta<\tau$. Hence by Proposition \ref{prop:cts_on_pesin_set} and Proposition \ref{prop:butler47}, along the orbit of $p^m$ the cocycle has block structure. Writing $u_m$ for the period of $p^m$, this implies that there is a frame $e_1,e_2$ at $p^m$ such that 
\begin{equation}\label{eqn:poincare_map_dim_2}
A^{u_m}(p^m)=\begin{bmatrix}
1 & a(m)\\
0 & 1
\end{bmatrix},
\end{equation}
for some number $a(m)$. Hence $e_1$ is carried back to itself by the return map. What we would like to know is that if we push forward $e_1$ by the cocycle, then when $\sigma^j(p^m)$ is near $y$ that $A^j(p^m)e_1$ is close to $E^s(y)$, the stable subspace of the periodic point $y$. This isn't too hard to show: near $y$, the cocycle is very close to the hyperbolic linear map $A(y)$. Near $y$ the cocycle preserves a cone $\mc{C}$ because it is close to a hyperbolic linear map. If $A^j(p^m)e_1\in \mc{C}$, then it will experience exponential growth. Keeping track carefully of norms, we obtain that by the time $e_1$ returns to $p^m$, that $\|A^{u_m}(p^m)e_1\|\ge \exp(\alpha m)$ for some $\alpha>0$. But this is impossible to have as $m\to \infty$ because $A^{u_m}(p^m)e_1=e_1$.
Thus if $A^j(p^m)e_1$ entered the cone $\mc{C}$, we would have a contradiction. 

The more difficult thing we must do is study what happens to the vector $e_2$. Our goal is the same: if $A^j(p^m)e_2$ is in $\mc{C}$, then this vector not only experiences exponential growth, but it experiences exponential growth \emph{transverse} to the $e_1$ direction. If we let $\Pi_{e_2}^{e_1}$ denote projection onto the $e_2$ subspace parallel to the $e_1$ subspace, then what we want to show is that $\|\Pi_{e_2}^{e_1}A^{u_m}(p^m)\|$ becomes unbounded.  We explain why this is reasonable. But first: let $j_0$ be the first time $\sigma^{j_0}(p^m)$ is near $y$ and let $j_1$ be the last time $\sigma^j(p^m)$ is near $y$ during its orbit. If we know that $A^{j_0}(p^m)e_2\in \mc{C}$, then we know that $A^{j_1}(p^m)e_2\in \mc{C}$ and has also experienced exponential growth and has most of its length in the direction $E^u(y)$. Hence we can consider its projection onto $e_1^{\perp}$ along the $e_1$ subspace, which we denote $\Pi_{e_1^{\perp}}^{e_1} A^{j_1}(p^m)e_2$. If at $\sigma^{j_1}(p^m)$ the vector $e_1$ is near $E^s(y)$ and $A^{j_1}(p^m)e_2$ is uniformly transverse to $e_1$, then $\|\Pi_{e_1^{\perp}}^{e_1}A^{j_1}(p^m)e_2\|$ is still exponentially large. Hence keeping track of the rest of the orbit carefully, we find that $\|\Pi_{e_2}^{e_1}A^{u_m}(p^m)e_2\|$ is growing in $m$ without bound. But this is impossible by \eqref{eqn:poincare_map_dim_2}, which implies that $\|\Pi_{e_2}^{e_1}A^{u_m}(p^m)e_2\|$ is uniformly bounded.

The thing we have not yet addressed is how we can conclude that $A^{j_1}(p^m)e_1$ is near $E^s(y)$. The idea is a trick: replace $A$ with $A^{-1}$ over the same shift $\sigma$. This new cocycle also takes values in a Zimmer block for the same measurable framing, but now the $E^s(y)$ and $E^u(y)$ subspaces are swapped. Hence at the point $j_1$ we must have that $A^{-(u_m-j_1)}(p^m)e_1$ is near $E^s(y)$ by iterating backwards and applying the same logic as in the paragraph containing equation \eqref{eqn:poincare_map_dim_2}. Thus by the previous paragraph we see that the orbit of $e_2$ makes uniformly small angle with $E^s(y)$ as well. By pushing this argument further, we can show that at the points $\sigma^{j_0}(p^m)$ and $\sigma^{j_1}(p^m)$ that every linear combination of $e_1$ and $e_2$ makes a small angle with $E^s(y)$, which is impossible.

In the case of higher dimensions and larger numbers of blocks, the proof is the same: one studies inductively block by block and shows for each term $\mc{E}_i$ that $\mc{E}_i$ stays near $E^c(y)$, the center subspace of $y$. We consider $E^c$ because in higher dimensions $y$ might not be a hyperbolic periodic point. One then obtains the same contradiction: a full dimensional subspace cannot have every vector in it make a small angle with a lower dimensional subspace.

The proof relies on the following shadowing lemma extracted from Butler's proof of his Proposition 4.1.

\begin{lem}\cite[Prop. 4.1]{butler2018measurable}\label{lem:sequence_of_points}
Suppose that $\Sigma,A$, and $\mu$ are as in Theorem \ref{thm:Zimmer_block_rigidity}. Suppose that there exists a periodic point $x$ such that $\lambda_{+}(x)=\lambda_-(x)=0$ and that there exists a periodic point $y$ such that $\lambda_+(y)-\lambda_-(y)>0$. Then for any sufficiently small $\chi>0$, there exists $N$, $\theta<\tau$, and a sequence of periodic points $p^m$ of period $u_m$ such that:
\begin{enumerate}
\item
$u_m=O(m)$;
\item
$\|A^{u_m}(p^m)\|\ge \exp(\chi m)$;
\item
$p^m$ lies in $\mc{D}(N,\theta)$.
\end{enumerate}
\end{lem}

Before proceeding we need to prove some additional estimates about the point $p^m$ and its orbit. Let us first describe the construction of the points $p^m$ and then indicate briefly why Lemma \ref{lem:sequence_of_points} holds.

Butler's construction proceeds as follows. We have a periodic point $x$ where $\lambda_+(x)=\lambda_-(x)=0$ and another periodic point $y$ with $\lambda_+(y)-\lambda_{-}(y)>0$. The idea is that we first shadow $x$ for some time, then shadow $y$, and then shadow $x$ again. This is then repeated periodically, which gives a periodic point $p^m$. By shadowing $x$ for a long time, the cocycle experiences little quasi-conformal distortion during that time, hence when the orbit comes near $y$, there is room for a little distortion without violating the inequalities defining $\mc{D}(N,\theta)$.

We now give a formal definition of $p^m$. Let $Q$ denote the $\{0,1\}$-matrix giving the valid words for $\Sigma$. Because $Q$ is irreducible, there exists some $M$ such that for all $m\ge M$, $Q^m$ has positive entries. Butler then shows that there exist two natural numbers $b$ and $c$ such that the following holds. We let $u_m=(2b+c+1)m$, which will be the period of the point $p^m$. We further assume that the period of $x$ and $y$ both divide $m$, so that $f^m(x)=x$ and $f^m(y)=y$. The word $p^m=(p^m_n)_{n\in \Z}$ is then defined by setting
\begin{align*}
p_j^m=x_j& \text{ for }-bm\le j\le bm,\\
p_j^m=y_{j-(b+1)m}& \text{ for }(b+1)m\le j\le (b+c+1)m.
\end{align*}
We define the $p^m_j$ indices where $bm<j<(b+1)m$ so that this substring connects $x_m$ to $y_0$ and is a valid word. This may always be done by the choice of $m\ge M$. Similarly, we define $p^m_j$ for $(b+c+1)m\le j\le (b+c+2)m$ so that this substring validly connects $y_{(b+c+1)m}$ to $x_0$. The remaining coordinates of $p^m$ are then defined by periodicity.

\begin{rem}\label{rem:b_and_c_large}
The numbers $b$ and $c$ appearing in the above construction are not arbitrary, however, we can assume that they are arbitrarily large. The precise place they are chosen appears in \cite{butler2018measurable} in the paragraph before equation (4.1).
\end{rem}

The following lemma will be used to approximate the cocycle along $p^m$ by the cocycle at the periodic point $y$ that $p^m$ shadows.

\begin{lem}\label{lem:inv_cones}
Suppose that $L\colon \R^d\to \R^d$ is a hyperbolic linear map of the form $A\oplus B\colon \R^k\oplus \R^{d-k}\to \R^k\oplus \R^{d-k}$ where $0<\|A^{-1}\|\le \mu^{-1}$, $\|B\|\le \lambda<\mu$.
For any $0<\sigma<1$, there exist $C,D,\epsilon_0,\ell>0$ such that for all $0<\epsilon<\epsilon_0$, any $D\epsilon\le \delta\le \epsilon_0$, and any sequence $L_1,\ldots,L_{j}$ of linear transformations of the form 
\[
L_i=\begin{bmatrix}
A+\alpha_i& \beta_i\\
\gamma_i & B+\delta_i
\end{bmatrix},\hspace{1cm}\max\{\|\alpha_i\|,\|\beta_i\|,\|\gamma_i\|,\|\delta_i\|\}\le \epsilon,
\]
the following hold:

\begin{enumerate}
\item

If we let $\mc{C}(\gamma)$ denote the cone:
\[
\{(u,v)\in \R^k\times \R^{d-k}\vert \|v\|\le \gamma\|u\| \},
\]
then the cone $\mc{C}(\delta^{-1})$ is invariant under each $L_i$.
\item
For any unit vector $v$ such that $\angle(v,\{0\}\times \R^{d-k})\ge \delta$, it follows that 
\begin{equation}
\|L_j\cdots L_1v\|\ge C(\mu-\epsilon^{1-\sigma})^j\delta.
\end{equation}
In fact, if we let $\Pi$ denote the orthogonal projection from $\R^d$ onto $\R^k\times \{0\}$, then
\begin{equation}
\|\Pi(L_j\cdots L_1v)\|\ge C(\mu-\epsilon^{1-\sigma})^j\delta.
\end{equation}
\item
If $j\ge \ell$, then any non-zero
\[
v\in L_j\cdots L_1(\mc{C}(\delta^{-1}))
\]
is uniformly transverse to $\{0\}\times \R^{d-k}$.

\end{enumerate}
\end{lem}

\begin{proof}
The proof of this lemma is standard: see, for instance, \cite[Lem.~6.2.10,11]{katok1997introduction}. We give it here as we need precise estimates on the cones, which are quite wide.

We begin by showing that there exists $D$ such that if $\delta>D\epsilon$ then $\mc{C}(\delta^{-1})$ is invariant under each $L_i$.
If we suppose that $(u,v)\in \mc{C}(\delta^{-1})$, then letting $L_i(u,v)=(u',v')$, we find that
\begin{align*}
\|u'\|&=\|Au+\alpha_iu+\beta_i v\|\\
&\ge \mu\|u\|-\epsilon \|u\|-\epsilon\delta^{-1}\|u\|\\
&=(\mu-\epsilon(1+\delta^{-1}))\|u\|.
\end{align*}
For $v'$ we similarly find
\begin{align*}
\|v'\|&=\|\gamma_iu+Bv+\delta_iv\|\\
& \le \epsilon\|u\|+\lambda\|v\|+\epsilon \|v\|\\
&=(\epsilon+\lambda\delta^{-1}+\epsilon\delta^{-1})\|u\|.
\end{align*}
So the cone $\mc{C}(\delta^{-1})$ is invariant if the following inequality holds:
\[
\epsilon+\lambda\delta^{-1}+\epsilon\delta^{-1}\le \delta^{-1}(\mu-\epsilon(1+\delta^{-1})).
\]
Multiplying through by $\delta$, we get
\[
\delta\epsilon+\lambda+\epsilon +\epsilon+\epsilon\delta^{-1}\le \mu,
\]
which certainly holds as long as $\epsilon_0$ is sufficiently small and $\delta$ is sufficiently big relative to $\epsilon$, which determines $D$. This establishes the invariance of the cone in (1).

We now show the second part of the conclusion. By the invariance of $\mc{C}(\delta^{-1})$, the following estimate will suffice. As before let $L_i(u,v)=(u',v')$. Fix $1>\sigma'>\sigma$. Then, for $(u,v)$ in the cone $\mc{C}(\epsilon^{-\sigma'})$ consider the estimate on $u'$ obtained above:
\begin{align*}
\|u'\|\ge(\mu-\epsilon-\epsilon\delta^{-1})\|u\|=(\mu-\epsilon-\epsilon^{1-\sigma'})\|u\|,
\end{align*}
the conclusion then follows because $\epsilon+\epsilon^{1-\sigma'}$ is lower order than $\epsilon^{1-\sigma}$ for any $0<\sigma'<\sigma$ as $\epsilon\to 0$.

We now show the final claim about vectors in $\mc{C}(\delta^{-1})$ becoming uniformly transverse. Note that if we compute in the same way as above, and let $L_i(u,v)=(u',v')$,
then for $(u,v)\in \mc{C}(\delta^{-1})$,
\[
\frac{\|v'\|}{\|u'\|}\le \frac{\epsilon+\lambda\delta^{-1}+\epsilon\delta^{-1}}{\mu-\epsilon-\epsilon\delta^{-1}}\le\frac{\epsilon+\lambda\delta^{-1}+\epsilon\delta^{-1}}{\mu-\epsilon-D^{-1}}.
\]
If we regard the right hand side of the above line as a linear map in $\delta^{-1}$, then this map has an attracting fixed point at 
\[
\delta^{-1}=\frac{\epsilon}{\mu-\lambda-2\epsilon-D^{-1}}.
\]
Hence if $j$ is sufficiently large then $L_j\cdots L_1\mc{C}(\delta^{-1})$ is uniformly transverse to $\{0\}\times \R^{d-k}$.
\end{proof}

We now record another abstract lemma for cocycles with an invariant bundle along an orbit. We write $\Pi_V^W$ for the projection onto $V$ parallel to a complementary subspace $W$.

\begin{lem}\label{lem:general_angle_decay}
Suppose that $\Omega$ is a compact metric space, $\sigma\colon \Omega\to \Omega$ is continuous, and $A\colon \Omega \to \GL(d,\R)$ defines a continuous linear cocycle. Then there exists $\lambda>0$ such that if $p\in \Omega$ and $A$ leaves invariant a bundle $\mc{E}(\sigma^j(p))\subseteq\R^d$ along an orbit segment of $p$ of length $n$, then for any vector $v$, we have
\[
\|\Pi_{\mc{E}^{\perp}}^{\mc{E}} A^n(p)v\|\ge e^{-n\lambda}\|\Pi_{\mc{E}^{\perp}}^{\mc{E}}v\|.
\]
\end{lem}
\begin{proof}
The angle between $A^nv$ and $\mc{E}$ can decay at most exponentially fast because the action induced by $A$ on $\RP^{d-1}$ is Lipschitz. As vectors contract in length at most exponentially fast, the conclusion follows.
\end{proof}

We now record a basic lemma:

\begin{lem}\label{lem:lipschitz_constant}
For each $d\in \N$ there exists a constant $C$ such that if $A\in \GL(d,\R)$, the Lipschitz constant of the induced action of $A$ on $\RP^{d-1}$ is at most $C\|A\|\|A^{-1}\|$.
\end{lem}
The proof of the above lemma is straightforward if one thinks of the double cover of $\RP^{d-1}$ by $S^{d-1}\subset\R^d$ and notes that the action of $A$ is linear. The constant $C$ is present to accomodate the reader's preferred choice of metric on $\RP^{d-1}$.

We now record a shadowing lemma that will be used below.

\begin{lem}\label{lem:exp_shadowing_lemma}
\cite[Lem.~4.9]{butler2018measurable}
Let $p\in \Sigma$ be a periodic point for $\sigma$ with extremal Lyapunov exponent $\lambda\coloneqq \lambda_+(p)$, and let $\epsilon>0$ be given. Then there is a constant $C\ge 1$ independent of $n$ such that for every $n\ge 1$, if $q\in \Sigma$ satisfies
\[
d(\sigma^j(p),\sigma^j(q))\le \max\{e^{-j\tau},e^{-(n-j)\tau}\}, \,\,\,\,0\le j\le n,
\]
then
\[
C^{-1}e^{n(\lambda-\epsilon)}\le \|A^n(q)\|\le Ce^{n(\lambda+\epsilon)}.
\]
\end{lem}

We use the above lemma to deduce the following.

\begin{lem}\label{lem:angles_near_conformal}
Suppose that $A\colon \Sigma\to \GL(d,\R)$ is a cocycle as in the statement of Theorem \ref{thm:Zimmer_block_rigidity} and that $p$ is a periodic point such that $\lambda_+(p)=\lambda_-(p)=0$. Let $\epsilon>0$ be given. There exists $C>0$ such that if $z$ is a point satisfying
\begin{equation}\label{eqn:exp_shadowing}
d(\sigma^j(p),\sigma^j(z))\le \max\{e^{-j\tau},e^{-(n-j)\tau}\}, \,\,\,\,0\le j\le n,
\end{equation}
and there exists an $A$-invariant subbundle $\mc{E}$ along the orbit segment $z,\ldots, \sigma^n(z)$, then for any vector $v$ 
\[
\|\Pi_{\mc{E}^{\perp}}^{\mc{E}}A^j(z)v\|\ge Ce^{-4\epsilon j}\|\Pi_{\mc{E}^{\perp}}^{\mc{E}}v\|, \text{ for }0\le j \le n.
\]
\end{lem}
\begin{proof}
Note that if $z$ satisfies \eqref{eqn:exp_shadowing} for $n$ then it also satisfies \eqref{eqn:exp_shadowing} for each $0\le n'\le n$ with $n'$ in the place of $n$.
Hence we may apply Lemma \ref{lem:exp_shadowing_lemma} to the the cocycle $A$ for  $0\le j\le n$, which gives that
\[
\max\{\|A^j(z)\|,\|(A^j(z))^{-1}\|\}\le Ce^{j\epsilon}.
\]
Thus we see by Lemma \ref{lem:lipschitz_constant}, that the angle between $A^j(z)v$ and $\mc{E}(\sigma^j(z))$ is at least $Ce^{-2j\epsilon}$ times the size of the angle between $v$ and $\mc{E}(z)$. Hence as the norm of $A^j(z)$ is bounded and cobounded by $Ce^{j\epsilon}$ the conclusion follows.
\end{proof}

We now record a lemma about linear algebra that is used below.

\begin{lem}\label{lem:proj_length}
Suppose $W$ is a subspace of $\R^d$. There exists $C>0$ such that if $V$ is any complement to $W$ and $q$ is uniformly transverse to $W$, then
\[
\|\Pi_{V}^Wq\|\ge C\|q\|.
\]
\end{lem}

\begin{proof}
The left hand side of the above inequality is minimized when $V=W^{\perp}$ because orthogonal projection onto $W$ minimizes the distance from $q$ to $W$. For $V=W^{\perp}$, the conclusion is immediate from transversality.
\end{proof}

When we refer to the ``angle" between two subspaces $V$ and $W$, we mean the smallest angle between a non-zero vector in each. For a periodic point $p$ of period $k$, we denote by $E^u(p),E^s(p)$, and $E^c(p)$, the stable, unstable, and center subspaces of its return map $A^k(p)$. Note that some of these subspaces may be trivial.

\begin{lem}\label{lem:subspaces_near_center}
Let $A,\Sigma$ and $\mu$ be as in Theorem \ref{thm:Zimmer_block_rigidity}. Let $\{0\}=\mc{E}_0\subset \mc{E}_1\subset \mc{E}_2\subset \cdots \subset \mc{E}_k=\R^d$ be the measurable flag given by the Zimmer block structure.
 There exists a choice of constants $b,c$ in Butler's construction and a small $\alpha>0$ such that if we define for each $m$:
\[
j_0=\lceil (1+\alpha)(b+1)m\rceil\text{ and }j_1=\lfloor(1-\alpha)(b+c+1)m\rfloor,
\]
then for all $\epsilon>0$ and for all $1\le i\le k$, there exists $M$ such that for all $m>M$, the bundles $\mc{E}_i(\sigma^{j_0}(p^m))$ and $\mc{E}_i(\sigma^{j_1}(p^m))$ both make angle at most $\epsilon$ with $E^c(p)$.
\end{lem}

\begin{proof}
Without loss of generality, we suppose that $x$ and $y$ are both fixed points. 
We proceed by induction. Suppose that the claim holds for $\mc{E}_i$, we check it for $\mc{E}_{i+1}$. Along the orbit of the point $p^m$ we fix an orthogonal complement $\mc{E}_i^{\perp}$ to $\mc{E}_i$ inside of $\mc{E}_{i+1}$ with respect to the usual metric on $\R^d$. Apply Lemma \ref{lem:inv_cones} to the map at $y$ with respect to the splitting $\R^d=E^u(y)\oplus(E^c(y)\oplus E^s(y))$, which gives constants $C,D,\epsilon_0,\sigma,\ell$ so that the conclusions of that lemma hold for these constants.

Note that the point $\sigma^{(b+1)m}(p^m)$ exponentially shadows $y$, and hence for all $j_0\le j\le j_1$, 
\[
d(\sigma^{j}(p^m),y)<e^{-\alpha b m}.
\]
So, for $0<\delta<\epsilon_0$ and sufficiently large $m$ by H\"older continuity of $A$ the cone $\mc{C}(\delta^{-1})$ is invariant under $A(\sigma^j(p^m))$ for $j_0\le j\le j_1$. Hence by Lemma \ref{lem:inv_cones}, for any $\epsilon>0$ and any unit vector $v$ that makes angle at least $\delta$ with $E^c(y)\oplus E^s(y)$:
\[
\|\Pi_{E^u(y)}^{E^s(y)\oplus E^c(y)}A^{j_1-j_0}(\sigma^{j_0}(p^m))v\|\ge C(\mu-\epsilon^{1-\sigma})^{j_1-j_0}\delta,
\]
and this holds for all $m$ greater than some fixed $M$. We would like to know the same thing but with $\mc{E}_i^{\perp}$ in the place of $E^u(y)$. But we are assured of this because, by the induction hypothesis, $\mc{E}_i(\sigma^{j_1}(p^m))$ is arbitrarily close to $E^c(y)$ for sufficiently large $m$, and $\mc{E}_i^{\perp}$ is uniformly transverse to $\mc{E}_i$. Hence by the third part of Lemma \ref{lem:inv_cones}, as $A^{j_1-j_0}(\sigma^{j_0}(p^m))v$ is uniformly transverse to $E^c(y)\oplus E^s(y)$, the conclusion follows by Lemma \ref{lem:proj_length}.

To summarize the above discussion, what we have shown is that for each $\epsilon,\delta>0$, there exists $M$ such that for all $m>M$, if $v\in \mc{C}(\delta^{-1})$ is a unit vector, then
\begin{equation}
\|\Pi_{\mc{E}_i^{\perp}}^{\mc{E}_i}A^{j_1-j_0}(\sigma^{j_0}(p^m))v\|\ge C(\mu-\epsilon^{1-\sigma})^{j_1-j_0}\delta.
\end{equation}

We now apply Lemma \ref{lem:exp_shadowing_lemma} and use the boundedness of the cocycle by $e^{\eta}$ for some $\eta>0$, which allows us to conclude that if $v$ is a unit vector at $p^m$ that is not in $\mc{E}_i(p^m)$, and $A^{j_0}(p^m)v$ lies in $\mc{C}(\delta^{-1})$, then
\[
\|\Pi_{\mc{E}_i^{\perp}}^{\mc{E}_i}A^{j_1}(p^m)v\|\ge Ce^{-\epsilon bm}e^{-(m+\alpha bm)\eta}e^{(j_1-j_0)\ln(\mu-\epsilon^{1-\sigma})}\delta.
\]
where the first and second exponentials come from first shadowing $x$ and then transitioning from shadowing $x$ to shadowing $y$, which takes $m$ steps.

Finally, the orbit of $p^m$ transitions to shadowing $x$ again and so by using Lemma \ref{lem:general_angle_decay} along the segment from $y$ to $x$ and then Lemma \ref{lem:angles_near_conformal} when we shadow $x$, 
we see that we get the following estimate on the full orbit, under the assumption that the image of $v$ at $\sigma^{j_0}(p^m)$ makes an angle of more than $\delta$ with $E^{s}(y)\oplus E^c(y)$:
\[
\|\Pi_{\mc{E}_i^{\perp}}^{\mc{E}_i}A^{(2b+c+2)m}(p^m)v\|\ge
Ce^{m(-\lambda-\eta-(5\epsilon b)-\alpha b\lambda-\alpha b\eta+(c-2\alpha)\ln(\mu-\epsilon^{1-\sigma}))}\delta.
\]
Note that $\delta$ is just a small fixed constant. Hence up to constants we have the lower bound
\begin{equation}\label{eqn:final_projection}
\|\Pi_{\mc{E}_i^{\perp}}^{\mc{E}_i}A^{(2b+c+2)m}(p^m)v\| \ge Ce^{m(-\lambda-\eta-(5\epsilon b)-\alpha b\lambda-\alpha b\eta+(c-2\alpha)\ln(\mu-\epsilon^{1-\sigma}))}.
\end{equation}
Because $p^m\in \mc{D}(N,\theta)$, we see that for any unit vector $v\in \mc{E}_{i+1}$, that 
\[
\|\Pi_{\mc{E}_i^{\perp}}^{\mc{E}_i}A^{(2b+c+2)m}(p^m)v\|
\] is uniformly bounded independent of $m$. This is because over $\mc{D}(N,\theta)$ the Zimmer block metric and the standard metric are uniformly comparable. Thus if the right hand side of equation \eqref{eqn:final_projection} is unbounded in $m$, for sufficiently large $m$ we obtain a contradiction of our assumption that $A^{j_0}(p_m)v$ made an angle of more than $\delta$ with $\mc{E}_i(\sigma^{j_0}(p_m))$. Hence we obtain a contradiction as long as 
\[
-\lambda-\eta-(5\epsilon b)-\alpha b\lambda-\alpha b\eta+(c-2\alpha)\ln(\mu-\epsilon^{1-\sigma})>0.
\]
But this is easy to achieve: for any choice of $b$ and $c$ sufficiently large, which we may have by Remark \ref{rem:b_and_c_large}, this follows as long as $\alpha$ and $\epsilon$ are both chosen to be sufficiently small. In particular, we obtain the contradiction of our assumption. Note as well that the specific sufficiently large $m$ that gives a contradiction depends only on $\delta$ and not on $v$.

To summarize, what we have just shown is that for all $\epsilon>0$, there exists $M$, such that for all $m>M$, $\mc{E}_{i+1}(\sigma^{j_0}(p^m))$ makes an angle of at most $\epsilon$ with $E^c(y)\oplus E^s(y)$. 

We now show that the same holds without needing the $E^c$ term. To do this we make an observation about the cocycle which is special to our setting. Consider what happens if we define a new cocycle $B\colon \Sigma\to \GL(n,\R)$ by the map $x\mapsto (A(x))^{-1}$ and use the original dynamics on the base. Note that for $B$ the stable and unstable bundle of $y$ are swapped and $x$ still satisfies $\lambda_+(x)=\lambda_-(x)=0$. Further, this new cocycle is also measurably conjugated to the same Zimmer block over the same measure $\mu$. This clear because if $A\colon V\to W$ is a linear transformation between two framed vector spaces whose representing matrix is in $U_0$, then $A^{-1}$'s matrix is also in $U_0$. Applying the above argument to $p^m$ for this cocycle, we can conclude something even stronger: for each $\epsilon>0$, there exists an $M$ such that for all $m>M$, and any $v\in \mc{E}_{i+1}(p^m)$, the vector $A^{j_0}(p^m)v$ makes an angle of at most $\epsilon$ with $E^c(p)$. We can reach an identical conclusion at the point $\sigma^{j_1}(p^m)$ by applying the same argument but instead iterating the cocycle backwards. Thus we have obtained the needed conclusion at the two endpoints. The rest of the result follows by induction.
\end{proof}

We now deduce Lemma \ref{lem:pps_have_0_exponent}.

\begin{proof}[Proof of Lemma \ref{lem:pps_have_0_exponent}.]
Suppose for the sake of contradiction that $y$ is a periodic point that does not satisfy $\lambda_+(y)=\lambda_-(y)=0$. We apply Lemma \ref{lem:sequence_of_points} to create a sequence of periodic points $p^m$ that satisfy the hypotheses of Lemma \ref{lem:subspaces_near_center}. Applying the conclusion of Lemma \ref{lem:subspaces_near_center} to the top term in the measurable flag, which is $\R^d$, implies that somehow all of $\R^d$ makes a small angle with the subspace $E^c(y)$, which is impossible. Thus we have obtained a contradiction of our initial assumption about $y$ and we are done.
\end{proof}

We now conclude the main theorem.

\begin{proof}[Proof of Theorem \ref{thm:Zimmer_block_rigidity}.]

By Lemma \ref{lem:pps_have_0_exponent} every periodic point $p$ has $\lambda_+(p)=\lambda_-(p)=0$. Thus applying \cite[Prop. 4.11]{kalinin2013cocycles}, we deduce that for any $\epsilon>0$, there exists $C_{\epsilon}$ such that for all $x\in \Sigma$ and $n\in \Z$,
\[
\|A^n(x)\|\cdot \|(A^n(x))^{-1}\|^{-1}\le C_{\epsilon} e^{\epsilon\abs{n}}.
\]
This gives immediately that there exist $N,\theta>0$ such that $\Sigma\subset \mc{D}(N,\theta)$. 
Thus by Proposition \ref{prop:cts_on_pesin_set}, we see that there exists a H\"older continuous flag 
\[
\{0\}=\mc{E}_0\subset\mc{E}_1\subset \cdots\subset \mc{E}_k=\R^d,
\]
and H\"older metrics on 
\[
\mc{E}_{i+1}/\mc{E}_i
\]
for all $0\le i<k$ that are invariant under the cocycle and quotient cocycles. 

By Lemma \ref{lem:piece_together_lemma}, we can then obtain a H\"older framing that gives the cocycle the desired form via an inductive process.
\end{proof}

\section{Measurable Conjugacies agree with Continuous conjugacies}

In this section we prove Theorem \ref{thm:transfer_rigidity}. In order to prove the rigidity result we ultimately seek, we will first prove the following restricted rigidity result that requires the conjugacy to take values in a Zimmer block.

\begin{prop}\label{prop:mbl_rigidity_in_a_block}
Suppose that $U_{\lambda}$ is a Zimmer block, $(\Sigma,\sigma)$ is a transitive subshift of finite type, $A,B\colon \Sigma\to U_{\lambda}$ are H\"older continuous functions defining cocycles over $\sigma$, and that $\mu$ is a fully supported measure with continuous product structure. Let $C\colon \Sigma\to U_{\lambda}$ be a $\mu$-measurable function such that for $\mu$-a.e.~$x\in \Sigma$,
\[
A(x)=C(\sigma(x))B(x)C(x)^{-1},
\]
then $C$ agrees $\mu$-a.e.~with a H\"older continuous function $\overline{C}$.
\end{prop}

In \cite[Thm.~5.1]{pollicott2001livsic}, it is shown that the above statement holds when $B=\Id$ and the measure is assumed to be an equilibrium state. In a later discussion in \cite[Sec.~9]{pollicott2001livsic}, it is remarked that the proof also works for cocycles taking values in nilpotent or compact groups when $B$ is non-trivial. To show how the structure of the Zimmer block is used, we give a proof of the above proposition along the general lines of the proof of \cite[Thm.~5.1]{pollicott2001livsic}.  Our proof is a bit longer than Pollicott and Walkden's because we must use the structure of $U_{\lambda}$.

\begin{proof}[Proof of Proposition \ref{prop:mbl_rigidity_in_a_block}.]
As before, it suffices to study only the case where $\lambda=0$. To begin, write the conjugacy $C$ as a block matrix with a $k$ by $k$ block structure given by the flag $\mc{E}_1\subseteq \cdots \subseteq \mc{E}_k$ preserved by $U_{\lambda}$. Let $C_{ij}$ denote the $i,j$th block. Note that all blocks with $i>j$ are zero matrices due to the Zimmer block structure.

As each $C_{ii}$ is an orthogonal matrix, by the measurable rigidity of cocycles taking values in compact groups \cite[Thm.~5.1]{pollicott2001livsic}, we see that $C_{ii}$ coincides almost everywhere with a H\"older continuous version $\overline{C}_{ii}$. Let $\hat{C}(x)$ be the block diagonal matrix with diagonal blocks $\overline{C}_{ii}(x)$. Then note that 
\begin{equation}\label{eqn:first_conjugation}
A(x)=(C(\sigma(x))\hat{C}(\sigma(x))^{-1})\hat{C}(\sigma(x))B(x)\hat{C}(x)^{-1}(\hat{C}(x)C(x)^{-1}).
\end{equation}
Thus we have reduced to the case that at almost every point each diagonal block of $C$ is an identity matrix. The proposition will follow by an induction using the following lemma.

\begin{lem}\label{lem:rigidity_for_2_by_2}
Suppose in the setting of Proposition \ref{prop:mbl_rigidity_in_a_block} that $U_{0}\subset \SL(d,\R)$ is a Zimmer block with two components and that $H$ is a $\mu$-measurable conjugacy between $A,B\colon \Sigma\to U_{0}$ of the form
\[
\begin{bmatrix}
\Id_a & C(x)\\
0 & \Id_b
\end{bmatrix}.
\]
Then $H$ coincides almost everywhere with a H\"older continuous conjugacy between $A$ and $B$.
\end{lem}
\begin{proof}
As $H$ is a measurable conjugacy, for $\mu$-a.e.~$x,x'$ and every $n$
\begin{equation}\label{eqn:conjugacy_product}
A^{-n}(x')A^n(x)=H(x')B^{-n}(x')H(\sigma^n(x'))^{-1}H(\sigma^n(x))B^n(x)H(x)^{-1}.
\end{equation}
We fix a Lusin set for $H$ of measure greater than $1/2$ and note that for any pair of points $x,x'$ that are regular for the Lusin set that these points simultaneously return to the Lusin set infinitely often at some sequences of times $n_i\to \infty$.

We now write out the right hand side of equation \eqref{eqn:conjugacy_product} in components:

\begin{align*}
&\begin{bmatrix}
\Id & C(x')\\
0 & \Id
\end{bmatrix}
\begin{bmatrix}
A_1^{-n}(x') & B_{-n}(x') \\
0 & A_2^{-n}(x') 
\end{bmatrix}
\begin{bmatrix}
\Id & -C(\sigma^n(x')) \\
0 & \Id
\end{bmatrix}\times \\
&\hspace{3cm}\begin{bmatrix}
\Id & C(\sigma^n(x))\\
0 & \Id
\end{bmatrix}
\begin{bmatrix}
A^n_1(x) & B_n(x) \\
0 & A^n_2(x)
\end{bmatrix}
\begin{bmatrix}
\Id & -C(x) \\
0 & \Id
\end{bmatrix},
\end{align*}
where we have written $B_n(x)$ for the upper right corner block of $B^n(x)$.

Multiplying these out we obtain that the upper right hand corner of this matrix is equal to:
\begin{align}\label{eqn:upper_right_corner_complicated}
\begin{split}
&-A_1^{-n}(x')A_1^n(x)C(x)+A_1^{-n}(x')B_n(x)+A_1^{-n}(x')(C(\sigma^n(x'))-C(\sigma^n(x)))A_2^{n}(x)\\&\,\,\,\,\,+B_{-n}(x')A_2^n(x)+C(x')A_2^{-n}(x')A_2^n(x),
\end{split}
\end{align}
which is also equal by assumption to the upper right hand corner of $A^{-n}(x')A^n(x)$, which equals:
\begin{equation}\label{eqn:upper_right_corner_simple}
A_1^{-n}(x')B_n(x)+B_{-n}(x')A_2^n(x).
\end{equation}
If we assume that $x'\in W^s_{loc}(x)$, then $\sigma^n(x')$ and $\sigma^n(x)$ approach exponentially fast.  Note now that because $A_1(x)$ and $A_2(x)$ are both isometric, that $A^{-n}_1(x')A_1^n(x)$ converges to a H\"older continuous function we denote by $H_{xx'}^1$ as in the discussion of holonomies in Section \ref{sec:preliminaries}. The same holds for $A_2(x)$, which gives a second holonomy $H_{xx'}^2$. Further note that if $n_i\to \infty$ is a sequence of times when both $x$ and $x'$ are in the Lusin set, we have that $C(\sigma^n(x))-C(\sigma^n(x'))\to 0$. Hence the third term in equation \eqref{eqn:upper_right_corner_complicated} also goes to zero as $n\to \infty$.
Consider the difference between equations \eqref{eqn:upper_right_corner_simple} and \eqref{eqn:upper_right_corner_complicated}, and take the limit as $n_i\to \infty$, this gives that
\[
-H_{xx'}^1C(x)+C(x')H^2_{xx'}=0.
\]
Hence,
\begin{equation}\label{eqn:formula_for_C}
C(x')=H_{xx'}^1C(x)(H^2_{xx'})^{-1}.
\end{equation}
Almost every point in $\Sigma$ is regular for the Lusin set, hence almost every pair of points $x,x'$ simultaneously return to the set along some subsequence. Thus by the product structure of $\mu$ we see that for almost every local stable leaf there is a H\"older continuous version of $C(x)$ over that leaf given by equation \eqref{eqn:formula_for_C}. The same holds for local unstable leaves. We are then done by Lemma \ref{lem:holonomy_inv_continuous_version}, which appears below.
\end{proof}

We now show how to conclude from Lemma \ref{lem:rigidity_for_2_by_2}. The idea is to induct up each superdiagonal of $C$. As the idea is simple and a formal induction would likely obscure the argument, we omit a general proof and illustrate the approach in the case of a Zimmer block with a flag of length four.

By equation \eqref{eqn:first_conjugation}, we have already reduced to the case where the diagonal blocks of $C$ are all the identity. The equation $A=C(\sigma(x))B(x)C(x)^{-1}$ then becomes:
\begin{align}
&\begin{bmatrix}
A_{11} & A_{12} & A_{13} & A_{14}\\
0 & A_{22} & A_{23} & A_{24}\\
0 & 0 & A_{33} & A_{34}\\
0 & 0 & 0 & A_{44}
\end{bmatrix}\\ \notag
&=\begin{bmatrix}
\Id & C_{12} & C_{13} & C_{14}\\
0 & \Id & C_{23} & C_{24}\\
0 & 0 & \Id & C_{34}\\
0 & 0 & 0 & \Id
\end{bmatrix}
\begin{bmatrix}
A_{11} & B_{12} & B_{13} & B_{14}\\
0 & A_{22} & B_{23} & B_{24}\\
0 & 0 & A_{33} & B_{34}\\
0 & 0 & 0 & A_{44}
\end{bmatrix}
\begin{bmatrix}
\Id & C_{12}' & C_{13}' & C_{14}'\\
0 & \Id & C_{23}' & C_{24}'\\
0 & 0 & \Id & C_{34}'\\
0 & 0 & 0 & \Id
\end{bmatrix},
\end{align}
where we write $C'_{ij}$ for the $ij$th block of $C(x)^{-1}$. Note that if we consider each $2\times 2$ block along the diagonal that these all satisfy the hypotheses of Lemma \ref{lem:rigidity_for_2_by_2}. From this we are able to conclude that the superdiagonal blocks $C_{12},C_{23}$ and $C_{24}$ of $C$ each coincide with a H\"older continuous version $\overline{C}_{12},\overline{C}_{23},\overline{C}_{34}$, respectively. As before we may then define a H\"older continuous $\hat{C}$ by
\[
\hat{C}(x)=
\begin{bmatrix}
\Id & \overline{C}_{12} & 0 & 0\\
0 & \Id & \overline{C}_{23} & 0\\
0 & 0 & \Id & \overline{C}_{34}\\
0 & 0 & 0 & \Id
\end{bmatrix}.
\]
Hence we may apply the same trick as in equation \eqref{eqn:first_conjugation}, and replace $B$ with the matrix $\hat{C}(\sigma(x))B(x)\hat{C}(x)^{-1}$ and the conjugacy $C$ with the conjugacy $C(x)\hat{C}(x)^{-1}$. Note that the conjugacy $C(x)\hat{C}(x)^{-1}$ has all $0$'s for its superdiagonal. Thus it suffices to study of the equation:
\begin{align}
&\begin{bmatrix}
A_{11} & A_{12} & A_{13} & A_{14}\\
0 & A_{22} & A_{23} & A_{24}\\
0 & 0 & A_{33} & A_{34}\\
0 & 0 & 0 & A_{44}
\end{bmatrix}&\\ \notag
&=
\begin{bmatrix}
\Id & 0 & C_{13} & C_{14}\\
0 & \Id & 0 & C_{24}\\
0 & 0 & \Id & 0\\
0 & 0 & 0 & \Id
\end{bmatrix}
\begin{bmatrix}
A_{11} & A_{12} & B_{13} & B_{14}\\
0 & A_{22} & A_{23} & B_{24}\\
0 & 0 & A_{33} & A_{34}\\
0 & 0 & 0 & A_{44}
\end{bmatrix}
\begin{bmatrix}
\Id & 0 & C_{13}' & C_{14}'\\
0 & \Id & 0 & C_{24}'\\
0 & 0 & \Id & 0\\
0 & 0 & 0 & \Id
\end{bmatrix}.
\end{align}
We now observe that the elements $C_{13}$ and $C_{24}$ satisfy a $2$ by $2$ equation as in Lemma \ref{lem:rigidity_for_2_by_2}. For instance, in the case of $C_{24}$, observe that:
\[
\begin{bmatrix}
A_{22} & A_{24}\\
0 & A_{44}
\end{bmatrix}
=
\begin{bmatrix}
\Id & C_{24}\\
0 & \Id
\end{bmatrix}
\begin{bmatrix}
A_{22} & B_{24}\\
0 & A_{44}
\end{bmatrix}
\begin{bmatrix}
\Id & C_{24}'\\
0 & \Id
\end{bmatrix}.
\]
Hence by applying Lemma \ref{lem:rigidity_for_2_by_2}, we conclude that there exists a H\"older continuous version of $C_{24}$. The same holds for $C_{13}$. Thus we may again define a matrix $\hat{C}$, and conjugate by this matrix as in \eqref{eqn:first_conjugation} to reduce to the case that $C_{13}$ and $C_{24}$ are also zero. Iterating this procedure once again allows us to conclude that $C_{14}$ is H\"older continuous and then we are done.

Exactly the same procedure works in general: one inductively improves the regularity of $C$ along the different super diagonals by studying $2$ by $2$ subsystems and applying Lemma \ref{lem:rigidity_for_2_by_2}.
\end{proof}

\begin{lem}\label{lem:holonomy_inv_continuous_version}
Suppose that $(\Sigma,\sigma)$ is a shift endowed with an invariant measure $\mu$ that has full support and continuous product structure. Suppose that there are uniformly H\"older continuous linear holonomy maps $H^s_{xy}, H^u_{xy}\in \GL(d,\R)$ defined on every local stable and unstable leaf and that these holonomies are transversely continuous. Suppose $C\colon \Sigma\to \R^d$ is a $\mu$-measurable map that is almost surely holonomy invariant in the sense that for $\mu$-a.e.~$x\in M$ and almost every $y\in W^u_{loc}(x)$, $H_{xy}^uC(x)=C(y)$ and the same holds for local stable manifolds. Then $C$ coincides almost everywhere with a H\"older continuous function $\overline{C}$ that is invariant under stable and unstable holonomies.
\end{lem}
\begin{proof}
We define a new function $\wt{C}$ that will equal $C$ almost everywhere. Let $\omega^1,\ldots, \omega^{\ell}\in \Sigma$ be words such that each $\omega^i_0$ is distinct, and these symbols exhaust those used to define $\Sigma$. Then if $\omega_0=\omega^i_0$, define $\wt{C}(\omega)$ by
\[
\wt{C}(\omega)=H^u_{[\omega,\omega^i]\omega}H^s_{\omega^i[\omega,\omega^i]}C(\omega^i).
\]
By the product structure, we can choose the $\omega^i$ such that $\wt{C}=C$ $\mu$-almost everywhere. From the definition, we see that $\wt{C}$ is uniformly H\"older along every unstable manifold and hence continuous by the transverse continuity of the holonomies. Next, observe that because $\wt{C}=C$ $\mu$-a.e.~that by holonomy invariance along stable manifolds and the continuity of $\wt{C}$, that $\wt{C}$ is invariant under stable holonomies along almost every stable manifold. As $\mu$ has full support, by transverse continuity of the holonomies, this implies that $\wt{C}$ is invariant under stable holonomies along every stable manifold. Thus by \cite[Thm.~19.1.1]{katok1997introduction}, we see that $\wt{C}$ is H\"older continuous so we are done.
\end{proof}

We now prove the measurable rigidity of a conjugacy to a Zimmer block without the assumption that the conjugacy takes values in a Zimmer block.

\begin{proof}[Proof of Theorem \ref{thm:transfer_rigidity}.]
To begin, the cocycle $B$, which takes values in a Zimmer block, preserves a flag $\mc{E}_i^B$, on which it acts isometrically on the intermediate quotients. If we push this flag forward by the conjugacy $C$, we obtain a flag $\mc{E}_i^A$ with the same properties that is invariant under $A$. The proof of Theorem \ref{thm:Zimmer_block_rigidity} shows that $\mc{E}_i^A$ actually coincides with a H\"older continuous $A$-invariant flag with H\"older continuous metrics on the intermediate quotients. Hence with respect to some H\"older continuous framing, we see that $A$ takes values in a Zimmer block. Thus we see that there is some H\"older continuous cocycle $D\colon \Sigma\to U_{\lambda}$ and a H\"older continuous conjugacy $H$ such that $HDH^{-1}=A$. Thus we see that $C(HDH^{-1})C^{-1}=B$. So, we see that $D$, which takes values in the same Zimmer block as $B$, and $B$ are measurably conjugate by the conjugacy $CH$. So, we have reduced to the case of a measurable conjugacy between two cocycles both taking values in Zimmer blocks. 
In fact, by our construction the conjugacy $CH$ also takes values in $U_{\lambda}$ $\mu$-almost everywhere. This is because $C$ carries the flag $\mc{E}_i^B$ to the flag $\mc{E}_i^A$, which is then carried to $U_{\lambda}$'s flag by $H$. The same also holds for the metrics on the intermediate quotients.
Thus we see that the conjugacy $CH$ takes values in a Zimmer block as it carries $U_{\lambda}$'s flag to itself and similarly preserves the metrics for $U_{\lambda}$ $\mu$-almost everywhere. Thus we are done by Proposition \ref{prop:mbl_rigidity_in_a_block}. 
\end{proof}

\bibliographystyle{amsalpha}
\bibliography{biblio.bib}

\end{document}